\documentclass[letterpaper]{amsart}
\usepackage{mathrsfs,amssymb}

\usepackage{graphicx}

\usepackage
[
pdfstartview={FitH},pdfview={FitH},
bookmarks={false}	
]%
{hyperref}

\theoremstyle{plain}
\newtheorem{theorem}{Theorem}
\newtheorem{lemma}[theorem]{Lemma}
\newtheorem{claim}[theorem]{Claim}

\newtheorem{corollary}[theorem]{Corollary}

\theoremstyle{remark} 

\newtheorem{remark}[theorem]{Remark}

\theoremstyle{definition} 
\newtheorem{question}[theorem]{Question}
\newtheorem{defn}[theorem]{Definition}
\newtheorem{statement}[theorem]{Statement}

\numberwithin{theorem}{section}

\DeclareMathOperator{\N}{\mathbb{N}}

\DeclareMathOperator{\ISigma}{I\Sigma}
\DeclareMathOperator{\RCA}{\mathsf{RCA}_0}
\DeclareMathOperator{\WKL}{\mathsf{WKL}_0}
\DeclareMathOperator{\ACA}{\mathsf{ACA}_0}

\DeclareMathOperator{\RT}{\protect\mathsf{RT}}

\DeclareMathOperator{\PRT}{\protect\mathsf{PRT}}
\newcommand{\paset}{P}
\newcommand{\zero}{\emptyset}

\newcommand{\cat}{^{\frown}}		
\newcommand{\upto}{\upharpoonright}	
\DeclareMathOperator{\dom}{dom}

\newcommand{\numhelpers}{\widehat{l}}	

\newcommand{\partitions}{S}
\DeclareMathOperator{\allones}{\overline{\mathbf{1}}}

\title[A packed Ramsey's theorem]{A packed Ramsey's theorem\\ and computability theory}
\author{Stephen Flood}
\date{\today}

\thanks{
	Partially supported by EMSW21-RTG 0353748, 0739007, and 0838506.  
	Much of this work appears in the author's Ph.D.\ thesis \cite{flood-thesis}, which was written at the University of Notre Dame under the direction of Peter Cholak.  
	Thanks also to Damir Dzhafarov for his helpful feedback on several drafts of this paper. 
	Special thanks to Wei Wang for his helpful suggestions on obtaining $low_2$ solutions to the stable form of packed Ramsey's theorem for pairs \cite{wang-egt}.
}



\begin{document}

\begin{abstract}
	Ramsey's theorem states that each coloring has an infinite homogeneous set, but these sets can be arbitrarily spread out.  Paul Erd\H os and Fred Galvin proved that for each coloring $f$, there is an infinite set that is ``packed together'' which is given ``a small number'' of colors by $f$.  
\par
	We analyze the strength of this theorem from the perspective of computability theory and reverse mathematics.  
	We show that this theorem is close in computational strength to standard Ramsey's theorem by giving arithmetical upper and lower bounds for solutions to computable instances. 
	In reverse mathematics, we show that that this packed Ramsey's theorem is equivalent to Ramsey's theorem for exponents $n\neq 2$.  When $n=2$, we show that it implies Ramsey's theorem, and that it does not imply $\ACA$.  
\end{abstract}

\maketitle

\section{Introduction}
\label{sect.Intro}

We begin with the standard definitions of Ramsey theory.
Fix any $X\subseteq\N$ and $n,k\in\N$.
We write $[X]^n$ to refer the set of $n$-element subsets of $X$.  That is, $[X]^n = \{Z\subseteq X : |Z|=n\}$.  
Given a coloring $f:[X]^n\rightarrow\{1,\dots,k\}$, we say $H$ is homogeneous for $f$ if $f$ assigns a single color to $[H]^n$.
Finite Ramsey's theorem says that for any $n,k,m\in\N$, there is some $w\in\N$ such that each coloring $f:[\{1,\dots,w\}]^n\rightarrow\{1,\dots,k\}$ has a size $m$ homogeneous set.
Similarly, infinite Ramsey's theorem says that for any $n,k\in\N$ and any coloring $f:[\N]^n\rightarrow\{1,\dots,k\}$, there is an infinite set $H\subseteq\N$ which is homogeneous for $f$.  

We will use the standard arrow notation for combinatorial bounds in finite Ramsey theory.  
In other words, given $w,m,n,k$ we write \emph{$w\rightarrow(m)^n_k$} to say that for each $X\subseteq\N$ with $|X|=w$ and for each coloring $f:[X]^n\rightarrow\{1,\dots,k\}$, there is a homogeneous $H\subseteq X$ with $|H|=m$.
In this notation, finite Ramsey's theorem says that for any $n,k,m\in\N$, there is a $w\in\N$ such that\ $w\rightarrow(m)^n_k$.

In arrow notation, infinite Ramsey's theorem simply asserts that $\N\rightarrow(\N)^n_k$ for each $n,k\in\N$.  
The infinite form of Ramsey's theorem has been studied extensively in computability theory and reverse mathematics, as in \cite{CJS,J72,liu}.
In the formal context of second-order arithmetic, we write
$\RT^n_k$ for the following sentence: 
$$\text{``}(\forall f:[\N]^n\rightarrow\{1,\dots,k\})(\exists H\subseteq\N)[H\text{ is infinite and is homogeneous for }f].\text{''}$$

Infinite Ramsey's theorem says that infinite homogeneous sets always exist, but infinite sets can be arbitrarily spread out.  
In \cite{ramseytype}, Erd\H os and Galvin use a function $\phi$ to describe how packed an infinite set is.  
\begin{defn}
Fix some $\phi:\N\rightarrow\N$. 
We say that $A\subseteq\N$ is \emph{packed for $\phi$} if $|A\cap\{1,\dots,w\}|\geq\phi(w)$ for infinitely many $w$.  
We say that $A\subseteq\N$ is \emph{sparse for $\phi$} if it is not packed for $\phi$.
\end{defn}

\begin{figure}[bth]
	\centering
	\includegraphics{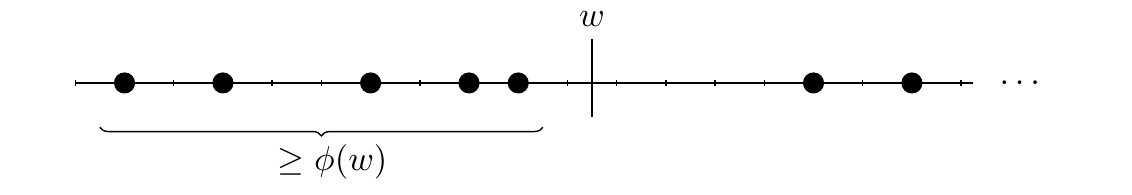}
	\caption{$A$ is packed for $\phi$ if $|A\cap\{1,\dots,w\}|\geq\phi(w)$ infinitely often.}
\end{figure}

This notion is interesting only when $\phi$ has $\liminf_w\phi(w)=\infty$, because otherwise any large enough finite set is packed for $\phi$.  
Unfortunately, Erd\H os and Galvin showed that each interesting $\phi$ has a coloring $f$ s.t.\ no homogeneous set is packed for $\phi$.

This is a consequence of the following, which is essentially Theorem 2.3 of \cite{ramseytype}.
\begin{theorem}[Erd\H os and Galvin \cite{ramseytype}]
\label{thm.EG.sharpcolors}
	Fix any $\phi:\mathbb{N}\rightarrow\mathbb{N}$ with $\liminf_w\phi(w)=\infty$ and any $n\in\mathbb{N}$, $n\geq 1$.  
There is a coloring $g:[\mathbb{N}]^n\rightarrow 2^{n-1}$ such that for any set $A$, either $A$ is sparse for $\phi$ or $A$ is given all $2^{n-1}$ colors by $g$.
\end{theorem}

Motivated by this, \cite{ramseytype} considers the following weakening of homogeneity.
\begin{defn}
Fix $n\in\N$.  A set $A$ is \emph{semi-homogeneous} for a coloring $f:[\N]^n\rightarrow\{1,\dots,k\}$ if $A$ is given at most $2^{n-1}$ colors by $f$.  That is, $A$ is semi-homogeneous for $f$ if $|\{f(Z) : Z\in[A]^n\}|\leq 2^{n-1}$.
\end{defn}

Using this weakening of homogeneity, Erd\H os and Galvin prove the following 
	variant of infinite Ramsey's theorem, which has a finite-Ramsey flavor.
\begin{theorem}[Erd\H os and Galvin \cite{ramseytype}]
\label{thm.PRTnk}
Fix $n,k\in\N$, and any $\phi:\N\rightarrow\N$ such that $w\rightarrow\big(\phi(w)\big)^{n}_{k+1}$ for all big enough $w$. 
For any $f:[\N]^n\rightarrow\{1,\dots,k\}$, there is a set $A$ which is packed for $\phi$ and semi-homogeneous for $f$.
\end{theorem}

Our goal in this paper is to study the strength of Theorem \ref{thm.PRTnk} from the perspective of computability theory and reverse mathematics.  
More precisely, we will study the computational strength required to produce packed semi-homogenous sets for any given computable $f$ and $\phi$.

We will classify computational strength in several ways.
First, recall that a formula $\theta$ is arithmetical if it has only number quantifiers.  
If $\theta$ is arithmetical with $n$-many alternations of $\forall$ and $\exists$, 
	recall that $\theta$ is $\Sigma^0_{n}$ if the outermost quantifier is $\exists$, 
	and that $\theta$ is $\Pi^0_n$ if the outermost quantifier is $\forall$.  
A set is $\Sigma^0_n$ ($\Pi^0_n$) if it has a $\Sigma^0_n$ ($\Pi^0_n$) definition,
	and it is $\Delta^0_n$ if it is both $\Sigma^0_n$ and $\Pi^0_n$.
Second, recall that $X$ is $low$ if $X'\equiv_T \emptyset'$ and that $X$ is $low_n$ if $X^{(n)}\equiv_T \emptyset^{(n)}$.
We will generally assume the reader is familiar with the basic definitions and results of computability theory.  
For a good introduction, see Part A of \cite{soare}.

Because of the close connections between computability and reverse mathematics, we will often discuss reverse mathematical corollaries and connections.
The basic idea of reverse math is to code theorems of mathematics inside second order arithmetic, and to compare their relative strength over the base system of $\RCA$. 
Intuitively, $T_1$ implies $T_2$ over $\RCA$ if you can prove $T_2$ using only $T_1$, computable constructions, and computable verifications.  
More formally, we say that $T_1$ implies $T_2$ over $\RCA$ if there is a proof of $T_2$ using only $T_1$, comprehension for $\Delta^0_1$ sets, induction for $\Sigma^0_1$ sets, and the axioms for ordered semi-rings.

We assume that the reader interested 
	in the details of the reverse mathematics results and arguments 
	is familiar with the area's standard definitions and techniques.
For a detailed and formal introduction to reverse mathematics, see Chapter 1 of \cite{simpson}. 
For a brief survey of reverse mathematics in Ramsey theory 
	and for more complete arguments see \cite{combprinciples}.
We begin with the usual notation.  

\begin{statement}
$\PRT^n_k$ is the assertion that
\begin{multline*}
	\text{``}\big(\forall \phi:\N\rightarrow\N\ s.t.\ (\forall w)[w\rightarrow(\phi(w))^n_{k+1}]\big)\big(\forall f:[\N]^n\rightarrow\{1,\dots,k\}\big)\\
	\big(\exists H\subseteq\N\big)\big[H\text{ is packed for }\phi\text{ and semi-homogeneous for }f\big].\text{''}
\end{multline*}
\end{statement}

In second-order arithmetic (when we are working over $\RCA$), $\PRT^n_k$ will refer to this $\Pi^1_2$ formula.  
When we are not working over $\RCA$, we will sometimes abuse this notation and write $\PRT^n_k$ to refer to Theorem \ref{thm.PRTnk} itself.

Because $\PRT^n_k$ is trivial when $\liminf_w\phi(w)<\infty$, 
	our proofs of $\PRT^n_k$ will always assume that $\liminf_w\phi(w)=\infty$.
In this case, any set $A$ which is packed for $\phi$ is automatically infinite.

\subsection{Outline}
We begin in Section \ref{sect.PRT-idea} with the basic proof idea for $\PRT^n_k$, 
	and with a proof that $\RT^1_k$ implies $\PRT^1_k$ over $\RCA$. 
In Section \ref{sect.PRT2}, we prove $\PRT^2_k$ using paths through a $\Pi^0_2$ definable tree.
In Section \ref{sect.PRT2.low2}, we adapt this proof to produce $low_2$ solutions to computable instances of $\PRT^2_k$.
In Section \ref{sect.tools-PRTn} we present the combinatorial tools which we use to prove $\PRT^n_k$, and in Section \ref{sect.PRTn} we prove $\PRT^n_k$ using paths through a $\Pi^0_n$ definable tree.

In Section \ref{sect.reversals}, we give lower bounds for the complexity of $\PRT^n_k$.  We begin by using $\PRT^n_{2^{n-1}-1+k}$ to prove to $\RT^n_k$ over $\RCA$.
Adapting this argument, we show that there is a computable instance of $\PRT^n_{2^{n-1}+1}$ that has no $\Sigma^0_n$ solution.

We summarize our results concerning the strength of $\PRT^n_k$ in Section \ref{sect.summary}.
Finally, we discuss several open questions about the strength of $\PRT^n_k$ in Section \ref{sect.questions}.

\subsection{Building packed semi-homgeneous sets}
\label{sect.PRT-idea}

Our proofs of $\PRT^n_k$ share a common method, 
	and the intuition from the early proofs is helpful in the later proofs.  

\begin{defn}
Suppose we have fixed $f$ and $\phi$ as in $\PRT^n_k$ for some $n,k$.
A finite set $Y\subset\N$ is a \emph{block} if it is $f$-homogeneous and there is $w\in\N$ such that $Y\subseteq\{1,\dots,w\}$ and $|Y|\geq\phi(w)$.
We say that a sequence of blocks $\{Y_i\}_{i\in I}$ is an \emph{increasing sequence of blocks} if $\max(Y_i)<\min(Y_{i+1})$ for each $i$.
\end{defn}
The main idea in each of these proofs is (1) to {explicitly} define helper colorings as paths through certain trees, (2) to use these colorings to define an increasing sequence of blocks, and (3) to refine this 
sequence to obtain the desired set.
For comparison, Erd\H os and Galvin define similar helper colorings using multiple ultrafilters.

We begin by showing that computable colorings of singletons have computable packed homogeneous sets.
In this proof, no tree is needed: we simply select and refine a sequence of blocks.
We will use the following claim to obtain this sequence of finite sets.

\begin{claim}[$\RCA$]\label{claim.PRT1.w} 
Suppose $\phi: \mathbb{N}\rightarrow\mathbb{N}$ satisfies $w\rightarrow(\phi(w))^1_{k+1}$ for all $w$.  
Then $(\forall m)(\exists w>m)[w-m\rightarrow (\phi(w))^1_k]$. 
\end{claim}
\begin{proof} 

Given $m$, take $w$ large enough so that $\phi(w)>m$.  
Fix any $A\subset\N$ with $|A|=w-m$, and any coloring $f:A\rightarrow\{1,\dots,k\}$. 
We must obtain a $\phi(w)$-element homogeneous set.
First, select any $X\subseteq\N$ of size $w$ such that\ $A\subset X$. 
Next, define  $\hat{f}:X\rightarrow\{1,\dots,k,k+1\}$ by setting $\hat{f}(x)=f(x)$ if $x\in A$, and setting $\hat{f}(x)=k+1$ if $x\notin A$.

Let $Y$ be a $\hat{f}$-homogeneous subset of $X$ of size $\phi(w)$.  
Then $Y$ is $\hat{f}$ homogeneous with color $c\in\{1,\dots,k\}$ since the color $k+1$ was assigned to $m<\phi(w)$ numbers. 
It follows that $Y\subseteq A$ is the desired $f$ homogeneous set of size $\phi(w)$. 
\end{proof}

\begin{theorem}[$\RCA$]
\label{thm.RT1k->PRT1k} 
For each $k\in\N$, $\RT^1_k$ implies $\PRT^1_k$ 
\end{theorem}
\begin{proof}
Fix $f:[\N]^1\rightarrow\{1,\dots,k\}$ and $\phi$ as in $\PRT^1_k$.  
We produce a set $A$ which is packed for $\phi$ and semi-homogeneous for $f$.  
Because $n=1$, `semi-homogeneous' means `homogeneous.'

Inductively define an increasing sequence $w_0<w_1<\dots$ by setting 
$w_0=1$ and $w_{i+1}$ to be the least $w>w_{i}$ such that $w-w_{i}\rightarrow (\phi(w))^1_k$.
By Claim \ref{claim.PRT1.w}, $w_{i+1}$ exists whenever $w_i$ exists. 
Notice that $w_i$ is defined by iterating a total $\Delta^0_1$ function $i$ many times.  
It follows that $i\mapsto w_i$ is total by $\Sigma^0_1$ induction (and Proposition 6.5 of \cite{combprinciples}).  
Furthermore, $\{w_i:i\in\N\}$ is unbounded by $\Sigma^0_1$ induction.  

For each $i$, let $Y_i\subseteq (w_i,w_{i+1}]$ be the $f$-homogeneous subset of size $\phi(w_{i+1})$ with least index as a finite set.  
For each $i$, $Y_i$ exists because $w_{i+1}-w_i\rightarrow (\phi(w_{i+1}))^1_k$.  
This sequence has a $\Delta^0_1$ definition.

The sequence $\{Y_i\}$ induces a coloring $g:\mathbb{N}\rightarrow \{1,\dots,k\}$ such that $g(i)$ is the color given to any/all $x\in Y_i$ by $f$.  Then $g$ is $\Delta^0_1$ because $f$ and the $Y_i$ are both $\Delta^0_1$, and $g$ is well defined because each $Y_i$ is $f$-homogeneous. 
Therefore, $\Delta^0_1$ comprehension proves that $g$ exists.
By $\RT^1_k$, there is a $c\in\{1,\dots,k\}$ and an infinite $H\subseteq\mathbb{N}$ such that $H$ is $g$-homogeneous with color $c$.  

Let $A=\bigcup_{i\in H} Y_i$.  
Clearly, $A$ is $f$-homogeneous.  
Furthermore, $Y_i\subset A$ for all $i\in H$, hence $|A\cap\{1,\dots,w_i\}| \geq \phi(w_i)$ for each $i\in H$.
Because $H$ is infinite, 
	we have that $A$ is packed for $\phi$ and homogeneous for $f$.  
\end{proof}

\begin{corollary}
\label{cor.PRT1.comp}
For each computable $f:\N\rightarrow\{1,\dots,k\}$ and each computable $\phi:\N\rightarrow\N$ such that $\phi(w)\leq \lceil \frac{w}{k+1}\rceil$ for all $w$, 
there is a computable set $A$ which is packed for $\phi$ and homogeneous for $f$.
\end{corollary}
\begin{proof}
Suppose that $f$ and $\phi$ in $\PRT^1_k$ are computable.
Then the sequence $\{Y_i\}$ defined in Theorem \ref{thm.RT1k->PRT1k} is uniformly computable, 
	and the infinite homogeneous set $H$ is computable given the finite parameter $c$.  
Thus $A$ is computable, as desired.  
\end{proof}

\section{A tree proof of $\PRT^2_k$}
\label{sect.PRT2}

\begin{defn}
Given $n,k\in\N$, a \emph{computable instance of $\PRT^n_k$} is a computable coloring $f:[\N]^n\rightarrow\{1,\dots,k\}$ and a computable $\phi:\N\rightarrow\N$ such that\ $w\rightarrow(\phi(w))^n_{k+1}$ for all $w$ and such that $\liminf_w\phi(w)=\infty$.
\end{defn}

The purpose of this section is to show that for each computable instance $f,\phi$ of $\PRT^2_k$,
	there is an infinite $\Pi^0_2$ tree s.t.\ 
	any path computes a set which is packed for $\phi$ and semi-homogeneous for $f$.  
We begin by fixing our notation for trees.

\begin{defn} 
Let \emph{$k^{<\N}$} denote the set of all functions $\tau$ such that
$\tau:\{1,\dots,w\}\rightarrow \{1,\dots,k\}$ for some $w\in\N$.
If $\dom(\tau)=\{1,\dots,w\}$, we will call $w=|\tau|$ the \emph{length} of $\tau$.
Given $\tau,\rho\in k^{<\N}$, we say that \emph{$\tau\preceq\rho$} if and only if $|\tau|\leq |\rho|$ and $\tau(x)=\rho(x)$ for each $x\in\{1,\dots,|\tau|\}$.
A set $T\subseteq k^{<\N}$ is a \emph{tree} if it is closed downward under $\preceq$.
Let $[T]$ denote the set of infinite paths through $T\subseteq k^{<\N}$. 
Then each $g\in[T]$ is a function $g:\N\rightarrow\{1,\dots,k\}$.
\end{defn}

We will define our helper colorings via initial segments, and we will use trees to organize these definitions.
Recall that we call a finite set $Y$ a block if it is $f$-homogeneous and there is $w\in\N$ such that $Y\subseteq\{1,\dots,w\}$ and $|Y|\geq\phi(w)$.

In proving Ramsey's theorem, 
	one builds infinite sets by adding one number at each step. 
We will build packed sets by adding one block $Y$ at each step.

During the construction, each block will be $f$ homogeneous, and all elements of our finite set $Y$ will be given a single color with all elements of future blocks.
Our goal is to build a semi-homogeneous (2-colored) set where each pair $x,y$ in the same block is given one single fixed color, and each pair $x,y$ in different blocks is given another (possibly different) fixed color.

\subsection{Largeness for exponent 2}
We will use a helper coloring $g:\N\rightarrow\{1,\dots,k\}$ to define this sequence of blocks.
When we select any block $Y$, we will commit to choosing all future blocks inside $\{y: (\forall x\in Y)[f(x,y)=g(x)]\}$.
By choosing each $Y$ to be $g$-homogeneous, we ensure that the elements of each $Y$ are given a single color with all future blocks.
To define $g$ so that this procedure can be iterated, we use a notion of ``largeness.''

The notion of largeness given by Erd\H os and Galvin is $\Pi^1_1$ 
	(quantifying over all possible $g:\N\rightarrow\{1,\dots,p\}$), and corresponds to our Claim \ref{claim.inductivestep}.  
To make the construction computable relative to some $\paset\gg \zero'$, 
	we use a related $\Pi^0_2$ definition of largeness.

\begin{defn}
\label{defn.large.2}
Fix a computable instance $f,\phi$ of $\PRT^2_k$.  
A set $X\subseteq\N$ is \emph{large} if 
\begin{align*}
(\forall m)(\forall p)(\exists w)(\forall \rho\in p^{w})[\exists Y&\subseteq (m,w]\cap X\ s.t. \\
	& |Y|\geq \phi(w),\\
	& Y \text{ is homogeneous for }f, \text{ and}\\
	& Y \text{ is homogeneous for }\rho.]
\end{align*}
We say $X$ is \emph{small} if $X$ is not large.
Note that ``$X$ is large'' is a $\Pi^{0,X}_2$ statement.
\end{defn}

Notice also that the definition of ``large'' and ``small'' depends on the computable instance $f,\phi$ of $\PRT^2_k$. 
Therefore, we will always use ``large'' in the context of some fixed computable instance. 

The proofs of Lemmas \ref{lemma.Nlarge} and \ref{lemma.unionlarge.2} 
	are very close to the corresponding proofs given in \cite{ramseytype}. 
We give full proofs for completeness.
 
We begin with the analog of the $n=2$ case of Claim 1 of \cite{ramseytype}.
\begin{lemma}\label{lemma.Nlarge}
Let $f,\phi$ be a computable instance of $\PRT^2_k$. Then
$\N$ is large.
\end{lemma}
\begin{proof}
Given $m$ and $p$, define $w$ large enough so that $(\ \phi(w)-m\ )\rightarrow (2)^1_{p}$.
For any $\rho\in p^{w}$, we must obtain $Y\subseteq(m,w]$ as in the definition of largeness. 

We will use the assumption in $\PRT^2_k$ that $w\rightarrow(\phi(w))^2_{k+1}$ for all $w$.
To do this, define a coloring $F:[\{1,\dots,w\}]^2\rightarrow \{1,\dots,k,k+1\}$ as follows: 
For $Z\in[\{1,\dots,w\}]^2$, set $F(Z)=f(Z)$ if $Z\subseteq(m,w]$ and $Z$ is $\rho$-homogeneous.  Otherwise, set $F(Z)=k+1$.  

Because $w\rightarrow(\phi(w))^2_{k+1}$, there is a set $Y\subseteq \{1,\dots,w\}$ such that $|Y|\geq \phi(w)$ and $Y$ is $F$ homogeneous for some $i\in\{1,\dots,k,k+1\}$.  

We show that $i\neq k+1$:  
Because $\phi(w)-m\rightarrow(2)^1_{p}$, there is an $2$-element subset $Z\subset Y\cap(m,w]$ which is $\rho$-homogeneous.  
By definition of $F$, $F(Z)\neq k+1$.
Because $Y$ is $F$ homogeneous, we see that $i=F(Z)\neq k+1$. 

Consequently, $Y$ is $f$-homogeneous, $Y\subseteq(m,w]$, and $|Y|\geq \phi(w)$. 
To see that $Y$ is $\rho$-homogeneous, notice that any 2-element subset of $Y$ is $\rho$-homogeneous.  
\end{proof}

Next, we give the analog of the $n=2$ case of Claim 2 of \cite{ramseytype}.
\begin{lemma}\label{lemma.unionlarge.2}
Let $f,\phi$ be a computable instance of $\PRT^2_k$.
The union of two small sets is small.
Therefore, for any finite partition $L=L_1\cup\dots\cup L_k$ of a large set $L$, one of the $L_i$ is large.
\end{lemma}
\begin{proof}
Given $S_1$ and $S_2$ small, fix $m_i$, $p_i$, and $w\mapsto \rho_{i,w}\in p_i^{w}$ witnessing the smallness of $S_i$.  
Define $m=\max\{m_1,m_2\}$ and $p=p_1\cdot p_2\cdot 2$. 
Note that $p>p_i$ (this is why Definition \ref{defn.large.n} quantifies over all possible choices of $p$).    
Define $s:\N\rightarrow\{1,2\}$ by $s(x)=1$ if $x\in S_1$, and $s(x)=2$ otherwise.  
Given $w$, define $\hat{\rho}_w(x)=\langle \rho_{1,w}(x), \rho_{2,w}(x), s(x)\rangle$ for each $x\leq w$.   

Suppose toward a contradiction that $S_1\cup S_2$ is large.  
Then there is some $\hat{w}$ witnessing that $S_1\cup S_2$ is large for $p$ and $m$ defined as above.  
Obtain the set $\hat{Y}\subseteq S_1\cup S_2$ promised by the definition of large applied to $m,p,\hat{w},$ and $\hat{\rho}_w$.  
Note that $\hat{Y}$ is homogeneous for $s$, so $\hat{Y}\subseteq S_i$ for some $i$. 
In either case, $\hat{Y}$ is contained in the interval $(m_i,\hat{w}]$, is homogeneous for $f$ and $\rho_{i,\hat{w}}$, and has size $|\hat{Y}|\geq \phi(\hat{w})$.  
This contradicts our choice of witnesses of the smallness of $S_i$. 
\end{proof}

\subsection{The construction}
To prove $\PRT^2_k$, we first show that if $X$ is large in the sense of Definition \ref{defn.large.2}, 
	it is large in the $\Pi^1_1$ sense used by Erd\H os and Galvin:  

\begin{claim} 
\label{claim.inductivestep}
Let $f,\phi$ be a computable instance of $\PRT^2_k$.
If $X$ is large and $g:\N\rightarrow\{1,\dots,p\}$ with $p\in\N$, then for each $m\in\N$ there is a $w\in\N$ and $Y\subseteq(m,w]\cap X$ such that $|Y|\geq\phi(w)$ and $Y$ is homogeneous for $f$  and $g$.
\end{claim}
\begin{proof}
Fix $X$ large, $g:\N\rightarrow\{1,\dots,p\}$, and any $m\in\N$.  
Find $w\in\N$ as in the definition of largeness.  
Then $g\upto w\in p^{w}$ so there is a set $Y\subseteq(m,w]$ with $|Y|\geq \phi(w)$ homogeneous for $g\upto w$ and $f$.  
Hence, $Y$ is homogeneous for $g$.
\end{proof}

We will use a single well-chosen helper coloring $g:\N\rightarrow\{1,\dots,k\}$ to build a packed semi-homogeneous set.

\begin{lemma}
\label{lemma.obtaing}
Let $f,\phi$ be a computable instance of $\PRT^2_k$.
There is an infinite $\Pi^0_2$ definable tree $T$ such that for each $g\in[T]$, and for all $w\in\N$, the set $\{y>w: (\forall x\leq w)[f(x,y)=g(x)]\}$ is large.
\end{lemma}
\begin{proof}
We begin by defining the $\Pi^0_2$ tree $T$.
For each $\tau\in k^{<\N}$, 
$$\tau\in T \iff [\{y> |\tau|: (\forall x\leq |\tau|)[\tau(x)=f(x,y)]\} \text{ is large}].$$

We show that $T$ is infinite by induction on $|\tau|$.  
The empty string is an element of $T$ by Lemma \ref{lemma.Nlarge}.  
Suppose $\tau\in T$.  
Then $\{y> |\tau|: (\forall x\leq |\tau|)[\tau(x)=f(x,y)]\}$ is large, so $\{y> |\tau|+1:\ f(|\tau|+1,y)= i \land  (\forall x\leq |\tau|)[\tau(x)=f(x,y)]\}$ is large for some $i\in \{1,\dots,k\}$ by Lemma \ref{lemma.unionlarge.2}.
Let $\tau\cat i$ denote the string obtained by adding the character $i$ to the end of the string $\tau$.  
Then $\tau\cat i\in T$.

Let $g\in[T]$ be any path.  
By the definition of $T$, the set $\{y> w: (\forall x\leq w)[g(x)=f(x,y)]\}$ is large for all $w\in\N$.  
In other words, $g$ is the desired helper coloring.
\end{proof}

We will be able to compute a solution to an instance of $\PRT^2_k$ 
	using any path $g$ through this tree $T$.  
The next lemma gives an upper bound on the computational strength of $g$.
Given $X,\paset\subseteq\N$, we say that $\paset\gg X$ if $\paset$ computes a path through each infinite $X$-computable binary tree.  

\begin{lemma}[Lemma 4.2 of \cite{CJS}]\label{lemma.choosePi2}
Suppose that $\paset\gg\zero'$ and that $(\gamma_{e,0},\gamma_{e,1})_{e\in\omega}$ is an effective
enumeration of all ordered pairs of $\Pi^0_2$ sentences of first order arithmetic. 
Then there is a $\paset$-computable $\{0,1\}$-valued (total) function $f$ such that $\gamma_{e,f(e)}$ is true whenever $\gamma_{e,0}\lor\gamma_{e,1}$ is true.
\end{lemma}

In words, Lemma \ref{lemma.choosePi2} says that any $P\gg\zero'$ can, if given two $\Pi^0_2$ sentences where at least one is true, select a true sentence.
It is straightforward to extend this lemma from choosing between $2$ sentences to choosing between $k$ sentences.
Given any $k$-many $\Pi^0_2$ formulas $\gamma_1,\dots,\gamma_k$ such that $\gamma_1\lor\dots\lor\gamma_k$ is true, we can use $\paset$ to uniformly find $c\in\{1,\dots,k\}$ such that $\gamma_c$ is true by querying the function $f$ from the above lemma $k-1$ times.  

In Section \ref{sect.PRT2.low2}, we will need the full strength of Lemma \ref{lemma.choosePi2}.
Here, we only need a straightforward consequence: 
	any set $\paset\gg \zero'$ can compute a path through each $\Pi^0_2$ definable $k$-branching tree.

\begin{theorem}
\label{thm.PRT2.pa}
Let $f,\phi$ be a computable instance of $\PRT^2_k$ and $\paset\gg\zero'$.
Then there is a set $A\leq_T \paset$ which is packed for $\phi$ and semi-homogeneous for $f$.
\end{theorem}
\begin{proof}

For the given computable instance $f,\phi$, apply Lemma \ref{lemma.obtaing} to obtain an infinite $\Pi^0_2$ definable tree. 
Then $[T]\neq\emptyset$ because $T$ is infinite.  
Because $T$ is $\Pi^0_2$ and because $P\gg\zero'$, there is a $\paset$-computable path $g\in[T]$.

We first select an increasing sequence of blocks $\{Y_i\}$ and an infinite set $\{w_0<w_1<\dots\}$ such that for each $i$, two properties hold: (1) $Y_i\subseteq (w_{i-1},w_i]$ with $|Y_i|\geq\phi(w_i)$ and (2) there is a color $c_i$ such that for each $j>i$, each $x\in Y_i$, and each $y\in Y_j$, we have  $f(x,y)=c_i=g(x)$.  

We proceed by induction on $s\in\N$.
Let $w_1$ and $Y_1$ be the number and finite set obtained by searching for the $w$ and $Y$ promised to exist by Claim \ref{claim.inductivestep} applied to the colorings $f,g$ and the large set $\N$ with $m=w_{0}=1$ and $p=k$.  

For the inductive step, suppose $Y_1,\dots,Y_s$ has been defined.
By our choice of $g$, the set $X=\{y>w_s: (\forall x\leq w_s)[f(x,y)=g(x)]\}$ is large.  
Let $w_{s+1}$ and $Y_{s+1}$ be the number and finite set obtained 
	by searching for the $w$ and $Y$ promised to exist 
	by Claim \ref{claim.inductivestep} applied to the colorings $f,g$ and the large set $X$ with $m=w_s$ and $p=k$.  
Note that $Y_{s+1}$ is $P$-uniformly computable.

Note also that $Y_{s+1}$ is homogeneous for $f$ and $g$, and $Y_{s+1}\subseteq(w_s,w_{s+1}]$ with $|Y_{s+1}|\geq \phi(w_{s+1})$.  
In other words, $Y_{s+1}$ is a block and property (1) holds for $i=s+1$.  
For each $i\leq s$, let $c_i=g(\min(Y_i))$.  
	We must show that property (2) holds for $j=s+1$.  
Because $Y_i$ is homogeneous for $g$, 
	and because $Y_{s+1}\subseteq X\subseteq\{y: (\forall i\leq s)(\forall x\in Y_i)[f(x,y)=g(x)]\}$, 
	we see that $f(x,y)=c_i=g(x)$ for each $x\in Y_i$ and each $y\in Y_{s+1}$.  
In other words, property (2) continues to hold and our construction produces the desired increasing sequence of blocks $\{Y_i\}$.

We next extract an infinite semi-homogeneous subsequence of blocks.
By property (2), there is a total function $f_1:\N\rightarrow\{1,\dots,k\}$ given by $f_1(i)=f(x,y)$ for any/all $x\in Y_i$ and $y\in Y_j$ for $j>i$.
Because each $Y_i$ is homogeneous for $f$, there is also a total function $f_2:\N\rightarrow\{1,\dots,k\}$ given by $f_2(i)=f(x,y)$ for any/all $x<y\in Y_i$.
Note that $f_1$ and $f_2$ are computable from $f$.

Applying the infinite pigeonhole principle twice, we obtain $I\subseteq\N$ infinite and homogeneous for $f_1$ and $f_2$. 
Furthermore, we can (non-uniformly) compute $I$ from $g$.  
Let $A=\bigcup_{i\in I} Y_i$.  

Then $A$ is semi-homogeneous because $f$ only assigns two colors to pairs in $A$: one to points in the same block and one to points in different blocks.
Because $I$ is infinite and because each $Y_i$ is a block, $A$ is packed.

Because this procedure was uniform in $g$, and because $g$ is $\paset$-computable, the set $A=\bigcup_{i\in I}Y_i$ is the desired $\paset$-computable packed semi-homogeneous set. 
\end{proof}

\begin{corollary} 
Each computable instance of $\PRT^2_k$ has a $\Delta^0_3$ definable solution. 
\end{corollary}
\begin{proof}
Recall that $\emptyset''$ is $\Delta^0_3$ and that $\emptyset''\gg\zero'$.
\end{proof}

The statement ``$\PRT^2_k$ has arithmetical solutions'' has a reverse math analog.

\begin{statement}
$\ACA$ is the axiom scheme which asserts that for each arithmetical formula $\phi(x,Y)$, if $Y$ is a set then $\{x\in\N: \phi(x,Y)\}$ exists as a set.
\end{statement}

\begin{corollary}
\label{cor.PRT2.ACA}
$\ACA$ implies $\PRT^2_k$ over $\RCA$.
\end{corollary}
\begin{proof}
Because $\ACA$ implies that $\zero''$ exists and because and $\zero''\gg\zero'$, the set constructions used to prove Theorem \ref{thm.PRT2.pa} can be performed in $\ACA$.  
We leave it to the reader to confirm that the verifications can be performed using induction for arithmetical formulas.
\end{proof}

Both the statement and the proof of Theorem \ref{thm.PRT2.pa} are the $n=2$ case of Theorem \ref{thm.PRTn.pa}.
In Section \ref{sect.PRT2.low2}, we will adapt this proof to obtain the $low_2$ proof of $\PRT^2_k$, and in Section \ref{sect.PRTn} we will generalize it to prove $\PRT^n_k$.

\section{A $low_2$ proof of $\PRT^2_k$}
\label{sect.PRT2.low2}

Our goal in this section is to prove that every computable instance of $\PRT^2_k$ has a $low_2$ solution.
As a consequence, we show $\PRT^2_k$ does not imply $\ACA$ over $\RCA$.

Our method builds on the work of Cholak, Jockusch, and Slaman in \cite{CJS}, who produced $low_2$ solutions to Ramsey's theorem for pairs. 
Their method was to use a degree $\paset\gg\zero'$ to build an infinite set (a $\Pi^0_2$ requirement) while simultaneously forcing the jump.
We will produce a $low_2$ packed semi-homogeneous set by using a degree $P\gg\emptyset'$ to compute a path through a sequence of $\Pi^0_2$ definable trees, while simultaneously forcing the jump. 

An important part of both constructions involves working $low$ trees.  
We say that $a$ is a \emph{lowness index} of $X$ if $X'=\Phi^{\emptyset'}_a$.
The Low Basis Theorem of Jockusch and Soare says that for each infinite computable tree $T\subseteq k^{<\N}$, there is an infinite $low$ path $g\in[T]$.  
Cholak, Jockusch, and Slaman note in \cite{CJS} that the uniformity in the proof of the Low Basis Theorem gives the following useful result:
there is a $\emptyset'$-computable uniform  procedure that takes any lowness index for any infinite $low$ tree $T$ and returns a lowness index for a path through $T$.

To ensure this uniformity during our construction, we will always implicitly associate a $low$ set with one of its lowness indices.
We will also use the standard observation that if $L$ is $low$, any statement $S(X)$ that is $\Pi^{0,L}_2$ is actually $\Pi^0_2$.

\subsection{The proof strategy}

To simplify the notation in this section, we fix a computable instance of $\PRT^2_k$.
That is, we fix a computable coloring $f:[\N]^2\rightarrow\{1,\dots,k\}$ 
	and a computable function $\phi:\N\rightarrow\N$ 
	such that $w\rightarrow(\phi(w))^2_{k+1}$ for all $w$.
Among other things, this allows us to say ``large'' to specify the definition of ``large for this computable instance.'' 

In the previous section, we built a sequence of blocks $\{Y_i\}$ such that each element of $Y_i$ was given color $g(\min Y_i)$ with each element of \emph{every} later block.  
In this section, we define a sequence of blocks $\{Y_i\}$ with a weaker property: each element of $Y_i$ will be given color $g(\min Y_i)$ with each element of \emph{almost every} later block.  

This will allow us to use Mathias forcing in the style of \cite{CJS} to build the sequence of blocks such that\ $\bigoplus_i Y_i$ is $low_2$. 
This induces a $low_2$ coloring of pairs: $d(i,j)=f(\min Y_i,\min Y_j)$.
Applying the following result, we will obtain an infinite $low_2$ semi-homogeneous sequence of blocks.

\begin{theorem}[Cholak, Jockusch, and Slaman \cite{CJS}]
\label{theorem.RT2.low2}
For each computable coloring $f:[\N]^2\rightarrow\{1,\dots,k\}$, 
there is an infinite $low_2$ homogeneous set.
\end{theorem}

A coloring $f:[\N]^2\rightarrow\{1,\dots,k\}$ is \emph{stable} if $\lim_y f(x,y)$ exists for each $x\in\N$.  
As we will see during the construction, the induced coloring $d$ is stable.  
Therefore, we only use the stable case of Theorem \ref{theorem.RT2.low2}.

Recall that $\{Y_i\}$ is an increasing sequence of blocks if for each $i$, 
	$Y_i\subseteq\N$ is $f$-homogeneous, $\max(Y_i)<\min(Y_{i+1})$, and $|Y_i\cap\{1,\dots,w\}|\geq\phi(w)$ for some $w$.

\begin{defn} 
For any $I\subseteq\N$, an increasing sequence of blocks $\{Y_i\}_{i\in I}$ is \emph{pre-semi-homogeneous for $f$} if for each $i,j\in I$ with $i<j$, there is a single color $c_{i,j}$  such that $f(x,y)=c_{i,j}$ for any $x\in Y_i$, and any $y\in Y_j$.
\end{defn}

During the construction, we will need a uniform way represent the sequence we are building.  
Furthermore, at the end of the construction, we will need to uniformly extract a subsequence.  
Because of this, we will work with a specific code for the sequence $\{Y_i\}$.

\begin{defn}
We say that a set $C$ is a code for an infinite sequence of sets $\{Y_i\}$ 
if there are sets $X$ and $W_X$ s.t.\ the following hold:
\begin{enumerate}
	\item $C=X\oplus W_X$,
	\item $W_X=\{w_0<w_1<\dots\}$ is an infinite set, and 
	\item $Y_i=X\cap (w_{i-1},w_i]$.
\end{enumerate}
\end{defn}

Intuitively, $W_X$ records boundaries between the blocks, and $X$ records the members of the blocks.
By forcing $X\oplus W_X$ to be $low_2$, we will ensure that we can uniformly recover the sequence ${\{Y_i\}}$ in a $low_2$ way.

\begin{lemma}
\label{lemma.uselow2code}
Suppose that $C\{Y_i\}=X\oplus W_X$ is a $low_2$ code for an increasing pre-semi-homogeneous sequence.  
Then there is a $low_2$ set $A$ that is packed for $\phi$ and semi-homogeneous for $f$.  
\end{lemma}
\begin{proof}
Let $C\{Y_i\}= X\oplus W_X$ be the code for an infinite pre-semi-homogeneous sequence of blocks $\{Y_i\}_{i\in\N}$.  
Note that the sequence $\{Y_i\}$ 
	induces a coloring $d:[\N]^2\rightarrow\{1,\dots,k\}$, 
	where $d(i,j)=f(\min Y_i,\min Y_j)$ for each $\{i,j\}\in[\N]^2$.

The induced coloring $d$ is computable from $C\{Y_i\}$.
By Theorem \ref{theorem.RT2.low2} relativized to $C\{Y_i\}$, we obtain a homogeneous set $H$ such that $(H\oplus C\{Y_i\})''\leq_T (C\{Y_i\})''$.  

Note that $\bigcup_{i\in H} Y_i$ is a $H\oplus C\{Y_i\}$-computable packed set, with a single color that is assigned by $f$ to any pair which is not contained in a single block.  

For each $i\in H$, let $f_2(i)$ be the color assigned by $f$ to each/any pair of elements in $Y_i$. 
By the pigeonhole principle applied to $f_2$, we obtain an infinite $H\oplus C\{Y_i\}$-computable set $I\subseteq H$ such that $A=\bigcup_{i\in I} Y_i$ is packed and semi-homogeneous.  
Note that $A$ is $H\oplus C\{Y_i\}$-computable.
Because $C\{Y_i\}$ is $low_2$ and because $(H\oplus C\{Y_i\})''\leq_T (C\{Y_i\})''$, we see that $A''\leq_T \emptyset''$.  Thus $A$ is $low_2$, as desired.
\end{proof}

\subsection{Building a $low_2$ sequence of blocks}
We will build the desired sequence $\{Y_i\}$ by Mathias forcing.
For convenience, we define `pre-conditions' (which have computable definitions) and `conditions' (which are pre-conditions with $low$ sets that satisfy a certain $\Pi^0_2$ property).

Our pre-conditions have the form $(\tau,D,W_D,L)$ 
	where $\tau\in k^{<\N}$ is a string, $D$ and $W_D$ are finite sets, and $L$ is a (possibly infinite) set.

\begin{defn}[Pre-conditions]
Let $p=(\tau,D,W_D,L)$ with $W_D=\{w_0<w_1<\dots<w_l\}$.
For each $i$ such that\ $1\leq i\leq l$, set $Y_i:=D\cap(w_{i-1},w_i]$.  \par
We say that $p$ is a \emph{pre-condition} if 
(1) $D=\bigcup_{i\leq l} Y_i$,
(2) $\{Y_i\}_{i\leq l}$ is pre-semi-homogeneous for $f$,
(3) $|Y_i|\geq \phi(w_i)$, 
(4) $w_l\leq |\tau|<\min(L)$
(hence each $Y_i$ is in the domain of $\tau\in k^{<\N}$), and
(5) each $Y_i$ is $f$- and $\tau$-homogeneous.
\end{defn}

Informally, we use $\tau(x)$ to keep track of the color we have committed to assigning all large enough numbers with $x$. 
More formally, we will ensure that each element $x$ of a block that is added to $D$ at stage $i$ will be given the color $\tau(x)$ with each element of any block which is added to $D$ at any stage $j>i$.

During the construction, we will first choose a finite extension of $\tau$, then pick a finite number of blocks $Y_i$ which are homogeneous for $\tau$.
This is a key difference between the proof in this section and that in the previous section: 
here we interleave extending the initial segment of the helper function and extending the sequence of blocks $\{Y_i\}$.

Recall that $D\oplus W_D$ is the code for the finite sequence of blocks built so far.
The next definition says that 
	our later promises must extend our earlier promises,
	that we can only add new blocks $Y_i$ on to the end of the sequence of blocks built so far,
	that we can only add $Y_i\subseteq L$,
	and that we can only remove elements from $L$.

\begin{defn}[Extending pre-conditions]
Given any two pre-conditions  $p=(\tau,D,W_D,L)$ and $\hat{p}=(\hat{\tau},\hat{D},\hat{W}_{\hat{D}},\hat{L})$, 
We say that $\hat{p}$ \emph{extends} $p$, written $\hat{p}\sqsupseteq p$, if the following hold:
	$\hat{\tau}\succeq\tau$,
	$D\subseteq \hat{D}\subset D\cup L$ and 
			$W_D\subseteq \hat{W}_{\hat{D}}\subset W_D\cup L$, and
	$\hat{L}\subseteq L$.
\end{defn}

For convenience, we often write $p$ to denote $(\tau,D,W_D,L)$.  
By decorating $p$ with different hats, we will denote different pre-conditions.  
For example, $\hat{p}$ always denotes $(\hat{\tau},\hat{D},\hat{W}_{\hat{D}},\hat{L})$.

\begin{defn}[Conditions]
A pre-condition $p=(\tau,D,W_D,L)$ is a \emph{condition} if 
	$L$ is large, 
	$L$ is $low$, and 
	$L\subseteq \{y> |\tau|: (\forall x\leq |\tau|)[f(x,y)=\tau(x)]\}$. 
\end{defn}

When $f$ is not a stable coloring, there are often many ways to extend $\tau$.  
We will use the following tree to organize these options.

\begin{defn} Fix a pre-condition $p=(\tau,D,W_D,L)$.  We define the tree of possible extensions of $\tau$ (written $T^L$ for short) by 
$$\sigma\in T^L \iff \left\{y \in L: y>|\sigma|\land (\forall x\leq |\sigma|)[f(x,y)=\sigma(x)]\right\}\text{ is large}.$$
\end{defn}

Because the definition of ``large'' is $\Pi^0_2$, it follows that $T^L$ is $\Pi^{0,L}_2$ definable.  
Indeed, note that there is a single $\Pi^{0,L}_2$ definition with parameter $L$ that defines $T^L$ for any pre-condition $p$.

If $L$ is large, then $\lambda\in T^L$ and Lemma \ref{lemma.unionlarge.2} implies that $T^L$ has no dead ends.
Thus if $p=(\tau,D,W_D,L)$ is a condition, then $T^L$ is an infinite tree and $\tau\in T^L$. 

The construction of our $low_2$ pre-semi-homogeneous sequence has three modules, each discussing how to extend a given condition $p$.  
Two modules will allow us to extend $p$ to force $\Phi^{C\{Y_i\}}_e(e)$ to either converge or diverge.
The third module will allow us to extend $p$ by a single block.  
Applying the third module infinitely many times will ensure that $\{Y_i\}$ is an infinite sequence of blocks.

Given $e$ and $p$, we will use the following $\Pi^{0,L}_1$ class to determine whether to force convergence or divergence.
Recall that for any condition $p=(\tau,D,W_D,L)$, there is a unique $t$ such that $D=Y_1\cup\dots\cup Y_t$  and $W_D=\{w_0,\dots,w_t\}$ for some choice of $Y_i\subseteq (w_{i-1},w_i]$.

\begin{defn}[The $\Pi^{0,L}_1$ class] 
If $g\in k^{\N}$, then $g\in \mathcal{U}^p_e$  if and only if  $g\succ\tau$ and
\begin{align*}
(\forall l\geq 1)&(\forall w_{t+l}>\dots >w_t=|\tau|)(\forall Y_{t+1},\dots, Y_{t+l} \text{ s.t.\ }Y_{t+i}\subseteq(w_{t+i-1},w_{t+i}]\cap L)\\
	 [&\text{If } \{Y_{t+i}\}_{i\leq l} \text{ is pre-semi-homogeneous for } f \text{, and } \\
	  &\quad \text{for each }i\leq l,\ |Y_{t+i}|\geq \phi(w_{t+i})\text{ and } Y_{t+i} \text{ is } g \text{ homogeneous}, \\
	 &\text{then } \Phi^{(D\cup Y_{t+1}\cup\dots\cup Y_{t+l})\oplus( W_D\cup\{w_{t+1},\dots,w_{t+l}\})}_{e}(e)\uparrow].
\end{align*}
\end{defn}

For any condition $p$ and for any $e$, there are two possible cases.  
Either $\mathcal{U}^p_e$ contains a function $g$, or it $\mathcal{U}^p_e$ is empty.
If $\mathcal{U}^p_e$ contains any function $g$, 
	we will be able to extend $p$ to a condition $\hat{p}$ which forces forces divergence.

\begin{lemma}[Forcing divergence]
\label{lemma.PRT2.diverge}
Let $P\gg\zero'$, $e\in\N$, and let $p$ be a condition such that $\mathcal{U}^p_e\neq\emptyset$. 
Then we can $P$-uniformly extend $p$ to a condition $\hat{p}$ such that $(\forall \tilde{p}\sqsupseteq\hat{p})\ \Phi^{\tilde{D}\oplus \tilde{W}_{\tilde{D}}}_{e}(e)\uparrow$.  
\end{lemma}
\begin{proof}
For simplicity, we will write $\mathcal{U}=\mathcal{U}^p_e$.
Suppose that $\mathcal{U}\neq\emptyset$.
Recall that the condition $p$ has the form $p=(\tau,D,W_D,L)$ where $L$ is a $low$ set given together with a $low$-ness index. 
By the Low Basis Theorem, there is some $g\in\mathcal{U}$ which is $low$ over $L$.  
By the uniform proof of the Low Basis Theorem, $g$ can be found $\paset$-uniformly, 
	along with an index witnessing that $g$ is $low^L$. 
Because $L$ is $low$, $L\oplus g$ is $low$.

The sets $L\cap g^{-1}(c)$ for $c\in\{1,\dots,k\}$ partition the large set $L$. 
For each such $c$, $L\cap g^{-1}(c)$ is computable from $L\oplus g$, so is $low$.  

By Lemma \ref{lemma.unionlarge.2}, there is a $c\in\{1,\dots,k\}$ such that\ $L\cap g^{-1}(c)$ is large. 
This statement is $\Pi^{0,L\oplus g}_2$.  
Because $L\oplus g$ is $low$, this statement is $\Pi^0_2$.  
Therefore, because $\paset\gg\emptyset'$, we can $\paset$-uniformly select one of these sets which is large.  

Let $\hat{L}=L\cap g^{-1}(c)$ for the $c$ selected above,
and set $\hat{\tau}=\tau$, $\hat{D}=D$, and $\hat{W}_{\hat{D}}=W_D$.  
By our choice of $c$, $\hat{L} \subseteq \{y> |\tau|: (\forall x\leq |\tau|)[f(x,y)=\tau(x)]\}$, and $\hat{L}$ is large and $low$.  
Indeed, we can $P$-uniformly find a $low$-ness index for $\hat{L}$ using the $low$-ness index of $L\oplus g$ along with the definition of $\hat{L}$.
In summary, $\hat{p}=(\hat{\tau},\hat{D},\hat{W}_{\hat{D}},\hat{L})$ is a condition extending $p=(\tau,D,W_D,L)$.

The definition of $g\in\mathcal{U}$ and the $g$-homogeneity of $\hat{L}$ ensures that no future initial segment of $X\oplus W$ will cause $\Phi^{X\oplus W}_e(e)$ to converge.  
We have thus forced that $\Phi^{X\oplus W}_e(e)\uparrow$.  
\end{proof}

The other possibility is that $\mathcal{U}^p_e=\emptyset$.
In this case, we can extend $p$ to a condition that forces convergence.

\begin{lemma}[Forcing convergence]
\label{lemma.PRT2.converge}
Let $\paset\gg\zero'$ and let $p$ be a condition s.t.\ that $\mathcal{U}^p_e=\emptyset$.
Then we can $\paset$-uniformly extend $p$ to a condition $\hat{p}$ such that $\Phi^{\hat{D}\oplus \hat{W}_{\hat{D}}}_{e}(e)\downarrow$.
Thus, $(\forall\tilde{p}\sqsupseteq\hat{p})\ \Phi^{\tilde{D}\oplus \tilde{W}_{\tilde{D}}}_{e}(e)\downarrow$.
\end{lemma}
\begin{proof}
For simplicity, we will write $\mathcal{U}=\mathcal{U}^p_e$.
Suppose $\mathcal{U}=\emptyset$.
Recall that $T^L$ has no dead ends and that $\tau\in T^{L}$ because $(\tau,D,W_D,L)$ is a condition.
In particular, $T^L$ is infinite, and no path through $T^L$ is in $\mathcal{U}$.

Note that $T^{L}$ is $\Pi^0_2$ because $T^{L}$ is $\Pi^{0,L}_2$ and because $L$ is $low$.
Uniformly in any $\paset\gg\zero'$, we can compute longer and longer (comparable) strings in ${T^{L}}$ which extend $\tau$.  
Because $\mathcal{U}$ is empty, 
	we will eventually compute a string $\hat{\tau}\in{T^{L}}$, 
	a pre-semi-homogeneous sequence of blocks $Y_{t+1},\dots,Y_{t+l}\subset\N$, 
	and dividers $w_{t+1}<\dots<w_{t+l}\leq |\hat{\tau}|$ which witness 
	$\Phi_e^{(\cdots)\oplus(\cdots)}(e)\downarrow$.
Set $\hat{D}=D\cup Y_{t+1}\cup\dots\cup Y_{t+l}$ and $\hat{W}_{\hat{D}}=W_D\cup\{w_{t+1},\dots,w_{t+l}\}$.
Let $u$ be larger than all numbers appearing so far 
	(including the use of the computation and $|\hat{\tau}|$), 
	and set  $\hat{L}=L\cap\{y\geq |\hat{\tau}|: (\forall x\leq |\hat{\tau}|)[\hat{\tau}(x)=f(x,y)]\}\cap\{y:y\geq u\}$.  

$\hat{L}$  remains $low$ because $\{y\geq |\hat{\tau}|: (\forall x\leq |\hat{\tau}|)[\hat{\tau}(x)=f(x,y)]\}\cap\{y:y\geq u\}$ is a uniformly computable set.
Indeed, we can $P$-uniformly find an index witnessing the lowness of $\hat{L}$ by using the $low$-ness index of $L$ together with the uniform definition of $\hat{L}$.
$\hat{L}$ remains large because $\hat{L}=^{*} L\cap\{y> |\hat{\tau}|: (\forall x\leq |\hat{\tau}|)[\hat{\tau}(x)=f(x,y)]\}$, which is large by our choice of $\hat{\tau}\in T^L$. 
In short, $(\hat{\tau},\hat{D},\hat{W}_{\hat{D}},\hat{L})$ is a condition extending $(\tau,D,W_D,L)$.
We have thus made progress toward our infinite pre-semi-homogeneous sequence of blocks, and we have forced that $\Phi^{\hat{D}\oplus\hat{W}_{\hat{D}}}_e(e)\downarrow$.
\end{proof}

Our last lemma shows that you can extend any condition $p$ to a condition $\hat{p}$ where $\hat{p}$ has one more block than $p$.

\begin{lemma}[Adding one block]
\label{lemma.PRT2.onemore}
Let $\paset\gg\zero'$ and let $p$ be any condition.  
You can $\paset$-uniformly extend $p$ to a condition $\hat{p}$ with $\hat{D}\supsetneq D$ and $\hat{W}_{\hat{D}}\supsetneq W_D$.  
\end{lemma}
\begin{proof}
Let $p=(\tau,D,W_D,L)$ be any condition.
By definition, $L$  is large.  

Applying the definition of largeness with $p=k$ and $m=|\tau|$, gives a $w$ such that for any $\rho\in k^{w}$ there is a block $Y\subseteq(m,w]\cap L$ with $|Y|\geq\phi(w)$ which is homogeneous for $\rho$ and $f$. 

Because ${T^{L}}$ contains $\tau$ and has no dead ends, it contains a string $\hat{\tau}\succeq\tau$ of length $w$.  
Take $Y\subseteq(m,w]\cap L$ to be the block with $|Y|\geq\phi(w)$ that is homogeneous for $\hat{\tau}$ and $f$. 

Define $\hat{L} = L\cap\{y> |\hat{\tau}|: (\forall x\leq |\hat{\tau}|)[\hat{\tau}(x)=f(x,y)]\}$.  
This set is large by the definition of $\hat{\tau}\in T^L$. 
Note that $\hat{L}$ is $low$ because it computable from $L$.
Define $\hat{D}= D\cup Y$ and $\hat{W}_{\hat{D}}=W_D\cup\{w\}$. 
Then $\hat{p}=(\hat{\tau},\hat{D},\hat{W}_{\hat{D}},\hat{L})$ is the desired condition extending $p$.
\end{proof}

We can now prove the main theorem of the section.
Working relative to a set $B\subseteq\N$, we say that $X$ is $low_n^B$ if $(X\oplus B)^{(n)}\leq_T B^{(n)}$.

\begin{theorem}\label{thm.PRT2.low2}
Every computable instance of $\PRT^2_k$ has a $low_2$ solution.
\end{theorem}
\begin{proof}
Fix any computable instance $f,\phi$ of $\PRT^2_k$ and 
	recall that the above definitions and lemmas were for an arbitrary computable instance of $\PRT^2_k$. 

Fix any $low^{\emptyset'}$ set $\paset$ s.t.\ $\paset\gg\emptyset'$.
We define $C\{Y_i\}$ by induction on the stage $s\in\N$.  
We begin with $s=0$ by setting $p_0=(\lambda,\emptyset,\{0\},\N)$.

Let $s>0$.
At stage $s=2e+1$ we force the $e^{th}$ jump. 
Deciding which jump forcing lemma to apply requires asking if a $\Pi^{0,L}_1$ class is nonempty.  
This can be rephrased as a $\Pi^{0,L}_1$ question, which can be answered uniformly by $\paset$.
Applying Lemma \ref{lemma.PRT2.diverge} or \ref{lemma.PRT2.converge} as appropriate, 
	we obtain the desired $p_s\sqsupseteq p_{s-1}$.

At stage $s=2e$ with $e>0$, we add a block to ensure $\{Y_i\}$ has at least $e$-many blocks.
Applying Lemma \ref{lemma.PRT2.onemore}, we obtain the desired condition $p_s\sqsupseteq p_{s-1}$.

This defines a $\paset$-uniform sequence of conditions $(\tau_i,D_i,W_{D_i},L_i)$.
From these conditions, we can $\paset$-uniformly recover a code $C\{Y_i\}$ for a sequence $\{Y_i\}$.  
Furthermore, the construction ensures that $\paset$ can compute the jump of the $C\{Y_i\}$.
Because $\paset$ is low over $\zero'$, it follows that $\zero''$ can compute the double jump of $C\{Y_i\}$.  
In other words, $C\{Y_i\}$ is $low_2$.
Applying Lemma~\ref{lemma.uselow2code}, we obtain a $low_2$ set that is packed for $\phi$ and semi-homogeneous for $f$, as desired.
\end{proof}

In fact, the above construction relativizes to any set $B\subseteq\N$.  
That is, for each $B$-computable $f:[\N]^2\rightarrow\{1,\dots,k\}$ 
	and each $B$-computable $\phi$ as in $\PRT^2_k$, 
	there is a $low_2^{B}$ set $A$ which is packed for $\phi$ and semi-homogeneous for $f$.
We leave the straightforward process of relativizing the proof to the reader.  
Iterating this result, we obtain an $\omega$-model of $\PRT^2_k$ with only $low_2$ sets.

\begin{corollary}
\label{cor.PRT2.notimp}
For each $k\in\omega$, there is an $\omega$-model of $\RCA+\PRT^2_k$ that is not a model of $\ACA$.
\end{corollary}
\begin{proof}
Iterating and dovetailing the relativized version of Theorem \ref{thm.PRT2.low2}, 
	we can produce an $\omega$-model of $\RCA+\PRT^2_k$ consisting only of $low_2$ sets.
At each stage, we add every set computable in the packed semi-homogeneous set, 
	ensuring that $\Delta^0_1$-comprehension holds. 
$\Sigma_1$-induction holds because the first order part is $\omega$.
\end{proof}

\section{Tools for proving $\PRT^n_k$}
\label{sect.tools-PRTn}

Our ultimate goal is to prove in Section \ref{sect.PRTn} that every computable instance of $\PRT^n_k$ 
	has a solution computable from each $\paset\gg\zero^{n-1}$.
We lay the groundwork of that proof in this section by introducing several background notions.  
We first introduce the special types of trees we will use to define helper colorings.  
We then give the appropriate analog of largeness, and prove the basic largeness lemmas.

As in Section \ref{sect.PRT2}, we will use helper functions to prove $\PRT^n_k$.
When $n>2$ we will need helper functions that are colorings of $[\N]^a$ for $a\in\{1,\dots,n-1\}$.
As before, we will define these helper colorings via initial segments.   

\begin{defn}
Let $k^{[<\N]^a}$ denote the set of all partial functions $\tau$ such that
$\tau:[\{1,\dots,w\}]^a\rightarrow \{1,\dots,k\}$ for some $w\in\N$.
If $\tau:[\{1,\dots,w\}]^a\rightarrow \{1,\dots,k\}$, we will call $w=|\tau|$ the \emph{length} of $\tau$. 
Given $\tau,\rho\in k^{[<\N]^a}$, we say that $\tau\preceq\rho$ if and only if (1) $|\tau|\leq |\rho|$ and (2) $\tau(Z)=\rho(Z)$ for each $Z\in[\{1,\dots,|\tau|\}]^a$.
\end{defn}

\begin{remark}
We will sometimes refer to a string $\tau\in k^{[\{1,\dots,w\}]^a}$ when $w<a$.  
In this case, $\dom(\tau)=\emptyset$.  
This has the strange, but not serious, consequence that the empty string $\lambda\in k^{[<\N]^a}$ has length $0$, $1$, \dots, and $a-1$.
\end{remark}

A tree $T\subseteq\N^{<\N}$ is \emph{$X$-computably bounded} if $T$ is $X$-computable and if there is an $X$-computable function $l:\N\rightarrow\N$ such that for each $w$, $l(w)$ bounds the code for each string in $T$ of length $w$. 
It is a standard observation that any $\paset\gg X$ can compute a path through each infinite $X$-computably bounded tree $T$.

Although $k^{[<\N]^a}$ is not $k$-branching, there is a computable function that bounds the strings of any given length $w$.  
Thus each $k^{[<\N]^a}$ is computably bounded.  

\begin{remark}
There are ${\binom{w+1}{a}-\binom{w}{a}}$ many sets in $[\{1,\dots,w+1\}]^a$ that are not in $[\{1,\dots,w\}]^a$. 
Therefore, each string in $k^{[\{1,\dots,w\}]^a}$ has exactly $k^{\binom{w+1}{a}-\binom{w}{a}}$ immediate successors in $k^{[\{1,\dots,w,w+1\}]^a}$.  
\end{remark}

Our motivation for working with subtrees of $k^{[<\N]^a}$ is the natural correspondence between colorings $g:[\N]^a\rightarrow\{1,\dots,k\}$ and elements of $k^{[\N]^a}$.

Our goal is to build a sequence of blocks $\{Y_i\}$ so that the color of $Z\in[\bigcup Y_i]^n$ depends only on how $Z$ is partitioned by the $Y_i$.  
When $n=2$, we built this sequence with the aid of the single helper function 
	that assigned $x$ the color it would be given with all big enough $y$.
When $n>2$, we will need $2^{n-1}-1$ helper colorings.  	
Therefore, when we select $Y$ we will need to ensure 
	that it is homogeneous 
	for each of the helper colorings $g_1,\dots,g_{2^{n-1}-1}$.

When $n=2$, the helper function was a map of numbers and large sets were sets of numbers.
Now, the helper functions will be maps of (up to) $n-1$-element sets and our large sets will be subsets $[\N]^{n-1}$.   

\subsection{Definitions and lemmas}
In the construction, we will define a helper coloring of exponent ${r_1}$ for each ordered tuple $(r_1,\dots,r_j)$ such that\ $r_1+\dots+r_j=n$ and $j>1$. 
Fix some enumeration of these $2^{n-1}-1$-many tuples.

For clarity, we will write $\numhelpers=2^{n-1}-1$ for the number of helper colorings.
We will write $a_i$ to refer to $1^{st}$ component of the $i^{th}$ tuple in our enumeration (which will be the exponent of the $i^{th}$ helper coloring).  
We can define $a_1,\dots,a_{\numhelpers}$ using any listing of the tuples that define the helper colorings. 

The above discussion suggests a $\Pi^1_1$ notion of largeness (quantifying over possible choices of the $g_i$). 
To make our constructions as effective as possible, we work with the following $\Pi^0_2$ notion of largeness:

\begin{defn} 
\label{defn.large.n}
Fix an instance $f,\phi$ of $\PRT^n_k$.
A set $L\subseteq [\N]^{n-1}$ is \emph{large} if 
$$\begin{array}{l@{}l@{}l@{}l}
(\forall m&)(\forall p_1&,\dots&,p_{\numhelpers}\in\N)\\
	&(\exists w)(&\forall \rho_1&,\dots,\rho_{\numhelpers}\ s.t.\ \rho_i\in {p_i}^{[\{1,\dots,w\}]^{a_i}})\\
	&	& [\exists Y&\subseteq (m,w]\ with\ [Y]^{n-1}\subset L\ s.t. \\
	&	& 	&|Y|\geq \phi(w),\\
	&	& 	&Y \text{ is homogeneous for }f, \text{ and}\\
	&	& 	&Y \text{ is homogeneous for each }\rho_i.]
\end{array}$$
We say $L$ is \emph{small} if $L$ is not large.
When $f,\phi$ is a computable instance of $\PRT^n_k$, 
	note that ``$L$ is large'' is a $\Pi^{0,L}_2$ statement.
\end{defn}

The next lemma will allow us to use our helper functions to extract a sequence of blocks.  

\begin{claim}
\label{claim.PRTn-inductive}
Fix any computable instance $f,\phi$ of $\PRT^n_k$ and any $\numhelpers$-many colorings $g_i:[\N]^{a_i}\rightarrow\{1,\dots,p_i\}$.
Suppose that $L\subseteq[\N]^{n-1}$ is large and $m\in\N$.
Then there exists $w\in\N$ and $Y\subseteq (m,w]$ such that\ $[Y]^{n-1}\subset L$, $|Y|\geq\phi(w)$, and $Y$ is $f$- and $g_i$-homogeneous for each $i$. 
\end{claim}
\begin{proof}
Given $m$ and the $p_i$'s, let $w$ be as in the definition of largeness. 
Setting $\rho_i=g_i\upto w$ for each $i$, we obtain the desired set $Y$.
\end{proof}

In the next two lemmas, we verify that Definition \ref{defn.large.n} satisfies the two main properties of largeness: (1) the set of all $n-1$-element sets is large, and (2) any finite partition of a large set contains at least one large set. 
As before, our proofs are adaptations of the analogous proofs in \cite{ramseytype}, and are given here for completeness.

We begin with the analog of Claim 1 in \cite{ramseytype}.
\begin{lemma}
Fix any computable instance $f,\phi$ of $\PRT^2_k$.  
Then $[\N]^{n-1}$ is large.
\end{lemma}
\begin{proof}
Fix $m,p_1,\dots,p_{\numhelpers}\in\N$.
First we must select $w\in\N$.  
To help define $w$, we define numbers $w_1,\dots,w_{\numhelpers}$ by induction from $\numhelpers$ down to $1$.  
Let $w_{\numhelpers}\in\N$ be large enough such that\ $w_{\numhelpers}\rightarrow (n)^{a_{\numhelpers}}_{p_{\numhelpers}}$.
Beginning with $i={\numhelpers}-1$, and counting down until $i=1$, let $w_i\in\N$ be large enough such that\ $w_i\rightarrow (w_{i+1})^{a_i}_{p_i}$.
Finally, let $w\in\N$ be large enough such that\ $\phi(w)-m\geq w_1$.  

Given any $\rho_1,\dots,\rho_{\numhelpers}$ such that\ $\rho_i\in {p_i}^{[\{1,\dots,w\}]^{a_i}}$, 
	we must obtain the desired set $Y\subseteq (m,w]$.
Toward this end, we define an auxiliary coloring $F:[\N]^n\rightarrow\{1,\dots,k,k+1\}$ as follows.
We set $F(Z)=f(Z)$ if $Z$ is homogeneous for each $\rho_i$ and $Z\subseteq(m,w]$.
Otherwise, we set $F(Z)=k+1$. 

We now use the assumption in $\PRT^n_k$ that $w\rightarrow(\phi(w))^n_{k+1}$ for all $w$.
Take any $F$-homogeneous subset $Y\subseteq\{1,\dots,w\}$ with $|Y|\geq\phi(w)$.  
Such a set $Y$ exists because $w\rightarrow(\phi(w))^n_{k+1}$.  
We will show that $Y$ is homogeneous for $F$ with some color $i\in\{1,\dots,k\}$, and is therefore the desired set.

Because $|Y|=\phi(w)$, it is clear that $|Y\cap(m,w]| \geq \phi(w)-m \geq w_1$.
Beginning with $i=1$, and counting up until $i={\numhelpers}-1$, we see that there is a $w_{i+1}$-element subset of $Y\cap(m,w]$ which is homogeneous for $\rho_1,\dots,\rho_{i}$.
Finally, there is a $n$-element subset $Z$ of $Y\cap(m,w]$ which is homogeneous for $\rho_1,\dots,\rho_{{\numhelpers}-1},\rho_{\numhelpers}$.

Note that by the definition of $F$, that $F(Z)=f(Z)\in\{1,\dots,k\}$. 
Because $Z\in[Y]^n$, and because $Y$ is $F$-homogeneous, $Y$ is given color $c\neq k+1$ by $F$. 
It follows that $Y\subset(m,w]$ and that $Y$ is $f$ homogeneous.  
It also follows that each $V\in[Y]^n$ is homogeneous for $\rho_1,\dots,\rho_{\numhelpers}$.
Because the exponent of each of these maps is less than $n$, 
	$Y$ itself is homogeneous for each $\rho_i$.  
Clearly $[Y]^{n-1}\subseteq[\N]^{n-1}$, and $|Y|\geq\phi(w)$.  
In other words, $Y$ is the desired set. 
\end{proof}

The next lemma is the analog of Claim 2 in \cite{ramseytype}.
\begin{lemma}\label{lemma.unionlarge.n}
Fix any computable instance $f,\phi$ of $\PRT^2_k$.  
The union of any two small sets of $[\N]^{n-1}$ is small.
In particular, for any finite partition $L=L_1\cup\dots\cup L_s$ of a large set $L$, one of the $L_i$ is large.
\end{lemma}
\begin{proof}
Suppose that $S_1,S_2\subset[\N]^{n-1}$ are small.  We show that $S_1\cup S_2$ is small.

Let $m_1$,\ $p_1,\dots,p_{\numhelpers}\in\N$, and $w\mapsto \rho^{w}_i$ be chosen to witness the smallness of $S_1$ 
	(for each $w\in\N$, the $\rho^{w}_i\in p_i^{[\{1,\dots,w\}]^{a_i}}$ witness the failure of $w$ to satisfy the definition of largeness).  
Let $m_2,$\ $q_1,\dots,q_{l}\in\N$, and $w\mapsto \sigma^{w}_i$ (s.t.\ $\sigma^{w}_i\in q_i^{[\{1,\dots,w\}]^{a_i}}$) witness the smallness of $S_2$.  
Recall that by our choice of the $a_i$, $a_t=n-1$ for some $t\leq {\numhelpers}$.  

To test the largeness of $S_1\cup S_2$, we now define $m$ (the lower bound on $Y$) and the $p_i$ (the number of colors assigned by the $\rho_i$).
Define $m=\max\{m_1,m_2\}$, define $\hat{p}_t=p_t\cdot q_t \cdot 2$, 
and define $\hat{p}_i=p_i\cdot q_i$ for $i\neq t$. 
Note that $\hat{p}_i>p_i$.  
This is why Definition \ref{defn.large.n} quantifies over all possible choices of $p_i$.

We wish to define the $\hat{\rho}_i$ so that any set $Y$ 
	which is homogeneous for each $\hat{\rho}_i$ has $[Y]^{n-1}\subset S_c$ for $c=1$ or $2$.  
Recalling that $S_c\subseteq[\N]^{n-1}$, we define $s:[\N]^{n-1}\rightarrow\{1,2\}$ by $s(U)=1$ if $U\in S_1$, and $s(U)=2$ otherwise.

Given any $w$, we first define $\hat{\rho}^w_t$ by setting $\hat{\rho}^w_t(U)=\langle \rho^w_t(U), \sigma^w_t(U), s(U)\rangle$ for each $U\in[\{1,\dots,w\}]^{n-1}$. 
Then, for each $i\neq t$ we define $\hat{\rho}^w_i$ by setting $\hat{\rho}^w_i(U)=\langle \rho^w_i(U), \sigma^w_i(U)\rangle$ for each $U\in[\{1,\dots,w\}]^{a_i}$.

Toward a contradiction, suppose that $S_1\cup S_2$ is large.  
Fix $\hat{w}$ and $\hat{Y}$ witnessing that $S_1\cup S_2$ is large 
	with $m$, $\hat{p}_i$, and $\hat{\rho}^{\hat{w}}_i$ as defined above. 
Then $[\hat{Y}]^{n-1}\subseteq S_1\cup S_2$ and $\hat{Y}$ is homogeneous for the  $\hat{\rho}^{\hat{w}}_i$ defined above.  
Note that $\hat{Y}$ is homogeneous for $s$ (because it is homogeneous for  $\hat{\rho}^{\hat{w}}_t$) so $[\hat{Y}]^{n-1}\subseteq S_j$ for some $j\in\{1,2\}$. 
In either case, $\hat{Y}\subseteq(m_j,\hat{w}]$ and $|\hat{Y}|\geq \phi(\hat{w})$.
Furthermore, $\hat{Y}$ is homogeneous for $f$, each $\rho^{\hat{w}}_i$, and each $\sigma^{\hat{w}}_i$.  
This contradicts our choice of parameters to witness of the smallness of both $S_1$ and $S_2$. 
\end{proof}

Our last largeness lemma comes from the proof of Claim 4 of \cite{ramseytype}.
Essentially, it says that for any coloring $h$ of exponent less than $n$, most elements of a large set are $h$-homogeneous.

\begin{lemma}	
\label{lemma.large-n.colorings}
Fix any computable instance $f,\phi$ of $\PRT^2_k$.  
Suppose that $L\subseteq[\N]^{n-1}$ is large and $p\leq n-1$.
For any coloring $h:[\N]^p\rightarrow\{1,\dots,s\}$, the set $\{Z\in L: Z$ is $h$-homogeneous$\}$ is large. 
\end{lemma}
\begin{proof}
Let $E=\{Z\in L : (\exists D_1,D_2\in [Z]^p)[h(D_1)\neq h(D_2)]\}$.
Then $L$ is the union of $E$ and $\{Z\in L: Z$ is $h$ homogeneous$\}$, so one of these is large by Lemma \ref{lemma.unionlarge.n}.

Suppose toward a contradiction that $E$ is large. 
Because $\liminf_x\phi(x)=\infty$, and by the definition of large, there are arbitrarily large finite sets $Y$ such that $[Y]^{n-1}\subseteq E$.  
Take $Y$ such that $|Y|\rightarrow (n-1)^p_{s}$. 
Then there is some $Z\in[Y]^{n-1}$ which is $h$-homogeneous.
But then $Z\in E$ by our choice of $Y$, contradicting the definition of $E$.
\end{proof}

\section{A tree proof of $\PRT^n_k$}
\label{sect.PRTn}

The purpose of this section is to show that for any $n,k\in\omega$ and for any $\paset\gg\zero^{(n-1)}$, 
each computable instance of $\PRT^n_k$ has a $\paset$-computable solution.

To simplify the notation in this section, we fix a computable instance of $\PRT^n_k$.
That is, we fix a computable coloring $f:[\N]^n\rightarrow\{1,\dots,k\}$ 
	and a computable function $\phi:\N\rightarrow\N$  as in $\PRT^n_k$.

\begin{defn}
Let $\partitions$ be the set of all ways of partitioning $n$ numbers into disjoint intervals.  
In other words, $\partitions=\{(r_1,\dots,r_l): r_1+\dots+r_l=n\}$ where each $r_i>0$.
We say that $(r_1,\dots,r_l)$ has length $l$. 
For each $l$, let $\partitions_l=\{(r_1,\dots,r_t)\in \partitions: t=l\}$ and let $\partitions_{\leq l}=\bigcup_{j\leq l} \partitions_j$.  
That is, $\partitions_{\leq l}$ is the set of partition types of length up to $l$.
\end{defn}

As before, our goal is to define a sequence of blocks $\{Y_i\}$ 
	such that the color of any $Z\in [\bigcup_i Y_i]^n$ 
	depends only on how the $\{Y_i\}$ partition $Z$.  
\begin{defn}
Suppose we have fixed an increasing sequence of blocks $\{Y_i\}$.
For any $Z\in[\bigcup Y_i]^n$, 
	we say that $(r_1,\dots,r_s)$ is the \emph{partition type of $Z$}
	if there are $i_1<\dots<i_s$ such that $|Z\cap Y_{i_j}|=r_j$ for each $j\leq s$, and if $Z=\bigcup_{j\leq s} Y_{i_j}$.	
\end{defn}

There are $2^{n-1}$ elements in $\partitions$. 
If we can ensure that the color of an $n$-tuple depends only on its partition type, 
	we will have ensured that $X$ is semi-homogeneous.
The first step in building the required sequence of blocks is to define an appropriate collection of helper colorings.  

\begin{defn}
Given $l\geq 1$ and a set of $2^{l}$-many functions $\mathcal{F}$, 
	we say that $\mathcal{F}$ is \emph{a collection of length $\leq l$ helper colorings} 
	if there is one exponent $r_1$ coloring $f_{r_1,\dots,r_i}:[\N]^{r_1}\rightarrow\{1,\dots,k\}$ 
	for each $(r_1,\dots,r_i)\in \partitions_{\leq l}$, 
	and if $f_n=f$.
\end{defn}

Our goal is to define a collection of length $\leq n$ helper colorings, 
	one function for each partition type $(r_1,\dots,r_l)\in\partitions$.  
To specify the properties that this collection of helper colorings should have, we need three more definitions.

\begin{defn}
Given finite $U,Z\subset\N$, we say that $Z$ \emph{extends} $U$ if $U=Z\cap\{1,\dots,\max(U)\}$.  
That is, $Z$ extends $U$ if $U$ is an initial segment of $Z$.
\end{defn}

The intuition is this: for each $r_1$ element set $U\in[\N]^{r_1}$, 
	$f_{r_1,\dots,r_l}(U)$ is the color that we promise to give 
	any $n$ element set $Z\subset \bigcup Y_i$ 
		with partition type $(r_1,\dots,r_l)$ 
	that  extends $U$. 
	
We will proceed by induction on $l$, using the coloring $f_{r_1+r_2,r_3\dots,r_{l}}$ to define the coloring $f_{r_1,r_2,\dots,r_l}$.
Recall that for exponent $n$, largeness is defined for subsets of $\{Z : Z\in[\N]^{n-1}\}$.

\begin{defn}
Suppose we have fixed a collection $\mathcal{F}$ of length $\leq l$ helper colorings.
For any finite set $W\subset\N$ and any $Z\in[\N\backslash W]^{n-1}$, we say that \emph{$Z$ is good with $W$} if:
$$(\forall(r_1,{\scriptstyle\dots},r_j)\in \partitions_{\leq l})(\forall U\in[W]^{r_1})(\forall V\in[Z]^{r_2})
	[f_{r_1,r_2,{\scriptstyle\dots},r_j}(U)=f_{r_1+r_2,r_3,{\scriptstyle\dots},r_j}(U\cup V)]$$
\end{defn}

Intuitively, we wish to ensure that for each good $Z$ of partition type $(r_1,\dots,r_l)$, 
	the color promised to $Z$ when viewed an extension of $U$, 
		where $Z\backslash U$ has partition type $(r_2,\dots,r_l)$,
	is the same as the color promised to $Z$ when viewed as an extension of $U\cup V$, 
		where $Z\backslash\ U\cup V$ has partition type $(r_3,\dots,r_l)$.  

Unfortunately, our intuitions about the helper functions 
	refer to the sequence of blocks that we are trying to define. 
We will use largeness to define the helper colorings 
	without reference to any sequence of blocks.

\begin{defn}
A collection $\mathcal{F}$ of length $\leq l$ helper colorings is made up of \emph{compatible} helper colorings 
	if $\big\{Z: Z$ is good with $\{1,\dots,w\}\big\}$ is large for each $w\in\N$.
\end{defn}

\begin{lemma}
\label{lemma.PRTn.use-colorings}
Fix $n,k\in\omega$ and any computable instance $f,\phi$ of $\PRT^n_k$. 
If there is a $\paset$-computable collection of length $\leq n$ compatible helper colorings $\mathcal{F}$, then 
	there is a $\paset$-computable set $A$ which is packed for $\phi$ and semi-homogeneous for $f$.
\end{lemma}
\begin{proof}
Let $\mathcal{F}=\{f_{r_1,\dots,r_l} : (r_1,\dots,r_l)\in\partitions\}$ 
	be any $\paset$-computable collection of compatible helper colorings.  

We first show that $\paset$ computes an infinite sequence of blocks $\{Y_i\}$ such that the color of any $Z\in[\{Y_i\}]^n$ depends only on two things: (1) the smallest block that contains an element of $Z$ and (2) the partition type of $Z$.  

More precisely, we first show that $\paset$ computes an infinite sequence of blocks $\{Y_i\}$ such that for any $Z\in[\bigcup_i Y_i]^n$, if $(r_1,\dots,r_l)$ is the partition type of $Z$ and if $Z_1$ is the $r_1$ smallest elements of $Z$, then $f(Z)=f_{r_1,\dots,r_l}(Z_1)$.

We define the $Y_i$ by induction on $i$.

Suppose that we have defined $Y_1,\dots,Y_i$ and $w_0,w_1,\dots,w_i$.
The set of $Z\in[\N]^{n-1}$ that are good with $W=Y_1\cup\dots\cup Y_i$ 
	is clearly large because $Y_1\cup\dots\cup Y_i\subseteq\{1,\dots,w_i\}$ 
	and because $\big\{Z\in[\N\backslash W]^{n-1}: Z$ is good with $\{1,\dots,w_i\}\big\}$ 
	is large by the compatibility of $\mathcal{F}$. 

To define $Y_{i+1}$, we search for the first finite set $Y_{i+1}$ and number $w_{i+1}\in\N$ s.t.\ 
$$\begin{array}{l@{}l@{}l@{}l}
	&	& Y_{i+1}	&\subseteq (w_i,w_{i+1}]\text{ and each }Z\in[Y_{i+1}]^{n-1}\text{ is good with }Y_1\cup\dots\cup Y_i, \\
	&	& 	&|Y_{i+1}|\geq \phi(w_{i+1}),\\
	&	& 	&Y \text{ is homogeneous for }f_n=f, \text{ and}\\
	&	& 	&Y \text{ is homogeneous for }f_{r_1,\dots,r_l}\text{ for each }(r_1,\dots,r_l)\in \partitions\text{ s.t.\ }1<l\leq n.
\end{array}$$
By the definition of large and by Claim \ref{claim.PRTn-inductive}, we will eventually find $Y_{i+1}$ and $w_{i+1}$.  
Because we can $\paset$-uniformly determine if a given finite set satisfies this property, 
it follows that we have a uniformly $\paset$-computable definition of the sequence $\{Y_i\}$.

Now consider any $Z\in[\{Y_i\}]^n$ with partition type $(r_1,\dots,r_l)$.
We claim that $f(Z)=f_{r_1,\dots,r_l}(Z_1)$. 
To see this, for each $i\leq l$, let $Z_i\subseteq Z$ be the least $r_1+\dots+r_i$ elements of $Z$.
By construction of $\{Y_i\}$ and the definition of ``good with $W$'', 
	$$f_{r_1, r_2,\ \dots\ ,\ r_l}(Z_1)=f_{r_1+r_2,\ \dots\ ,\ r_l}(Z_{2}),\text{ and}$$ 
	$$(\forall i<l)\big[f_{r_1+\dots+r_i,\ r_{i+1},\ \dots\ ,\ r_l}(Z_i)=f_{r_1+\dots+r_i+r_{i+1},\ \dots\ ,\ r_l}(Z_{i+1})\big].$$ 
Inductively, we see that $f_{r_1,\dots,r_l}(Z_1)=f_n(Z_l)$.  
Recall that $f_n=f$ because $\mathcal{F}$ is a collection of helper colorings, and therefore $f_{r_1,\dots,r_l}(Z_1)=f(Z)$, as desired.
Note that $Z_1\subseteq Y_i$ for some $i$, and that each $Y_i$ is $f_{r_1,\dots,r_l}$ homogeneous.
It follows that the color of $Z\in[\{Y_i\}]^n$ depends only on (1) and (2).

To obtain a subsequence where the color of $Z$ depends only on its partition type, notice that the sequence of blocks $\{Y_i\}$ induces one coloring for each partition type.  
For each $(r_1,\dots,r_l)\in S$, define $h_{r_1,\dots,r_l}:\N\rightarrow\{1,\dots,k\}$ by setting $h_{r_1,\dots,r_l}(i)=f_{r_1,\dots,r_l}(Z)$ for any/all $Z\in[Y_i]^{r_1}$.
Iterating the infinite pigeonhole principle, once for each of the $2^{n-1}$-many induced colorings, we get an infinite set $I$ homogeneous for each $h_{r_1,\dots,r_l}$. 
Note that we can (non-uniformly) compute $I$ from $\paset$.
 
Define $A=\bigcup_{i\in I} Y_i$.  
Clearly $A\leq_T \paset$. 
Because $A$ is the union of infinitely many blocks, $A$ is packed for $\phi$.  
Note that the color given to any $Z\in[A]^n$ is completely determined by the way that $Z$ is partitioned by $\{Y_i\}_{i\in I}$.   
In other words, $A$ is the desired packed semi-homogeneous set.
\end{proof}

\subsection{Obtaining the helper colorings}
\label{PRTn.colorings}

In this subsection, we will 
	(1) describe how we will build trees using colorings, 
	(2) prove these trees are infinite, and 
	(3) show how to build a compatible collection of helper colorings using a specific choice of $\Pi^0_n$ trees.
In particular, we will define a collection of trees 
	so that any set of paths through these trees are the helper functions 
	$f_{r_1,\dots,r_l}:[\N]^{r_1}\rightarrow\{1,\dots,k\}$.

For $l=1$, we simply define $T_1=\{\sigma: \sigma\prec f\}$ so that $[T_1]=\{f\}=\{f_n\}$.  
We will define the remaining trees by induction on $l$. 

In principle, we could define one tree $T_{r_1,\dots,r_l}\subseteq k^{[<\N]^{r_1}}$ for each partition type $(r_1,\dots,r_l)\in S$.  
Then $f_{r_1,\dots,r_l}:[\N]^{r_1}\rightarrow\{1,\dots,k\}$ will be some path through this tree.
Unfortunately, we must define these trees so that the resulting functions are compatible.
This requires defining each tree relative to all of the helper colorings selected so far.
Because the definition of ``large'' is $\Pi^0_2$, the trees are $\Pi^0_2$ relative to their parameter 
	and the resulting procedure would require a degree $\paset\gg\,\zero^{(2^{n-1})}- 1\,$.
To reduce this complexity, we will instead define many trees simultaneously. 

For each $l$, we define a single tree 
		whose paths define \emph{all} the colorings $f_{r_1,\dots,r_t}$ for $(r_1,\dots,r_t)\in\partitions_l$.
More precisely, we define a single tree whose elements are a direct sum of strings of the following form for some $w\in\N$
	$$\tau=\Big(\bigoplus_{(r_1,\dots,r_l)\in \partitions_l} \tau_{r_1,\dots,r_l}\Big) \in \Big(\bigoplus_{(r_1,\dots,r_l)\in \partitions_l} k^{[\{1,\dots,w\}]^{r_1}}\Big).$$
We say this $\tau$ has length $w$ because each component $\tau_{r_1,\dots,r_l}$ is defined on exactly the subsets of $\{1,\dots,w\}$.
We now give the formal definition of the tree. 
\begin{defn}
\label{defn.PRTn.Tl}
Fix any collection of length $\leq l$ helper colorings $\mathcal{F}$ with $l\geq 2$.
We define $T_l^{\mathcal{F}}$ 
	to be the subtree of $\bigcup_{w\in\N} \left( \bigoplus_{\dots} k^{[\{1,\dots,w\}]^{r_1}}\right)$
	obtained  by setting
$$\Big(\bigoplus_{(r_1,\dots,r_l)\in \partitions_l}\tau_{r_1,\dots,r_l}\Big) \in T_l^{\mathcal{F}} $$
\emph{if and only if} there is a large set of $Z\in[\N\backslash\{1,\dots,w\}]^{n-1} \text{ s.t.}$
	\begin{enumerate}
	\item $Z$ is good with $\{1,\dots,w\}$ for the colorings being defined: 
		\begin{flushleft}
		$\displaystyle(\forall (r_1,\dots,r_l)\in \partitions_l)\big[ 
		(\forall\,U\in[\{1,\dots,w\}]^{r_1})(\forall\,V\in[Z]^{r_2})$
	\end{flushleft}
	\begin{center}
		$[\tau_{r_1,r_2,\dots,r_l}(U)=f_{r_1+r_2,r_3,\dots,r_l}(U\cup V)]\big]$
	\end{center}
	\item and $Z$ is good with $\{1,\dots,w\}$ for the colorings in $\mathcal{F}$: 
	\begin{flushleft}
		$(\forall (r_1,\dots,r_{m})\in \bigcup_{m<l}\partitions_{m})(\forall\,U\in[\{1,\dots,w\}]^{r_1})(\forall\,V\in[Z]^{r_2})$\\
	\end{flushleft}
	\begin{center}
		$\big[f_{r_1,r_2,\dots,r_m}(U)=f_{r_1+r_2,r_3,\dots,r_m}(U\cup V)\big]$
	\end{center}
	\end{enumerate}
\end{defn}

When it is clear which collection of length $\leq l$ colorings is being used to define $T_l^{\mathcal{F}}$, 
	we will often simplify our notation by writing $T_l=T_l^{\mathcal{F}}$.

\begin{claim}
\label{claim.Tl-inf}
	Fix $l\geq 2$ and any compatible collection of length $\leq l$ helper colorings $\mathcal{F}$.  
	Then $T_{l}^{\mathcal{F}}$ is infinite.  
\end{claim}
\begin{proof}
Fix $w\in\N$. We will show that $\rho\in T_l=T_l^{\mathcal{F}}$ for some $\rho$ of length $w$.  
Consider 
	$$\rho =\bigoplus_{(r_1,\dots,r_l)\in \partitions_l} \rho_{r_1,\dots,r_l} \in \bigoplus_{(r_1,\dots,r_l)\in \partitions_l}k^{[\{1,\dots,w\}]^{r_1}}.$$

By definition, $\rho\in T_l$ if and only if there is a large set of $Z\in[\N\backslash\{1,\dots,w\}]^{n-1}$ 
	which respect the promises that $\rho$ and $\mathcal{F}$ make about finite subsets of $\{1,\dots,w\}$.
Unfortunately, for any given $Z$, there may be some $(r_1,\dots,r_l)\in \partitions_l$ and some $U\in[\{1,\dots,w\}]^{r_1}$ such that $Z$ is not even homogeneous for $V\mapsto f_{r_1+r_2,r_3,\dots,r_l}(U\cup V)$.  
In this case, $Z$ is not good with $\{1,\dots,w\}$ for \emph{any} string $\rho$.  
  
Because $\mathcal{F}$ is compatible, the set $G=\{Z: Z$ is good with $\{1,\dots,w\}$ for the collection $\mathcal{F}$ is large. 
We claim that the set 
	$G\cap\big\{Z\in[\N\backslash\{1,\dots,w\}]^{n-1} : Z$ is good with $\{1,\dots,w\}$ for \emph{some} $\rho$ with $|\rho|=w\big\}$ 
	is large.  
Because $\partitions_l$ and $\{1,\dots,w\}$ are finite, there are finitely many functions $V\mapsto f_{r_1+r_2,\dots,r_l}(U\cup V)$.  
Iterating Lemma \ref{lemma.large-n.colorings} (once for each function) yields a large set of $Z\in G$ such that 
	for each $(r_1,\dots,r_l)\in \partitions_l$ and each  $U\in[\{1,\dots,w\}]^{r_1}$, there is a color $c$ such that $(\forall V\in[Z]^{r_2})[f_{r_1+r_2,r_3,\dots,r_l}(U\cup V)=c]$.  
Letting $\rho_{r_1,r_2,r_3,\dots,r_l}(U)$ be the corresponding $c$, 
	we see that this $Z$ respects the promises made by this $\rho$, as desired.

The set of all $\rho$ of length $w$ induces a partition of this large set into the finitely many sets 
	$G\cap\big\{Z: Z$ is good with $\{1,\dots,w\}$ for $\rho\big\}$.  
By Lemma \ref{lemma.unionlarge.n}, 
	one of the $G\cap\big\{Z: Z$ is good with $\{1,\dots,w\}$ for $\rho\big\}$ is large; 
	thus the associated string $\rho$ is an element of $T_l$.
Because $w$ was arbitrary, we have shown that $T_l$ contains a string of each length $w$.  
Thus, $T_l$ is infinite.
\end{proof}

Fix any path $p\in[T_l]$.  Then $p$ will have the form 
$p=\bigoplus_{(r_1,\dots,r_{l})\in \partitions_{l}} f_{r_1,\dots,r_{l}}$.
Intuitively, the $(r_1,\dots,r_l)^{th}$ component of $p$ will be the helper function $f_{r_1,\dots,r_l}:[\N]^{r_1}\rightarrow\{1,\dots,k\}$.
More formally, we have the following observation.

\begin{remark}
Fix $l<n$ and any collection $\mathcal{F}_l$ of length $\leq l$ compatible helper colorings.
If $p\in[T^{\mathcal{F}_l}_l]$ and if $f_{r_1,\dots,r_{l}}$ is the $(r_1,\dots,r_{l})$-th component of $p$, 
	then $\mathcal{F}_{l+1}=\mathcal{F}_l\cup\{f_{r_1,\dots,r_{l}}: (r_1,\dots,r_{l})\in S_l\}$ 
	is a collection of length $\leq l+1$ compatible helper colorings.
\end{remark}

We will define the trees recursively.  
That is, we will use the paths through $T_1,\dots,T_l$ to define the tree $T_{l+1}$. 
Because the definition of ``large'' is $\Pi^0_2$, the tree $T_l$ is $\Pi^0_2$ relative to the parameter $\mathcal{F}$.

We will use the next observation to reduce the complexity of these trees.
For convenience, we give a proof taken from the first half of Proposition 12 of \cite{RKL}.

\begin{claim}
\label{claim.Pi2-to-Sigma1}
For any $\Pi^{0,X}_2$ tree $T$, there is a $\Sigma^{0,X}_1$ tree $S$ such that $[T]=[S]$.
\end{claim}
\begin{proof}
Fix a $\Pi^0_2$ definable tree $T$.  Then there is a formula $\phi$ which is $\Delta^0_1$ such that $\tau\in T\ \leftrightarrow\ (\forall y)(\exists z)\phi(\tau,y,z)$.
	Using the $\Delta^0_1$ formula $\psi(\tau, \hat{z}) =_{def} (\forall x,y\leq |\tau|)(\exists z<\hat{z})\phi(\tau\upharpoonright x,y,z)$, 
we can define a $\Sigma^0_1$ tree $S$ by $\tau\in S$ $\leftrightarrow$ $(\exists \hat{z}) \psi(\tau,\hat{z})$.
Then $[S]$ $=$ $\{f: (\forall w)(\exists \hat{z})\psi(f\upharpoonright w,\hat{z}) \}$ 
$=$ $\{f: (\forall x)(\forall y)(\exists z)\phi(f\upharpoonright x,y,z) \}$
$=$ $[T]$. 
\end{proof}

We can now prove $\PRT^n_k$ using an arithmetical procedure.
Recall that the in this section, we have worked with an arbitrary fixed computable instance $f,\phi$ of $\PRT^n_k$.

\begin{theorem}
\label{thm.PRTn.pa}
Fix any $n,k\in\omega$ and any $\paset\gg \zero^{(n-1)}$.
Each computable instance of $\PRT^n_k$ has a $\paset$-computable solution.
\end{theorem}
\begin{proof}
Fix a computable instance $f,\phi$ for $\PRT^n_k$.
We begin by showing that 
	for any $\paset\gg\zero^{(n-1)}$, 
	there is a $\paset$-computable collection of length $\leq n$ compatible helper colorings $\mathcal{F}_n$.

During this construction, 
	we will need a uniform way to represent collections of length $\leq l$ compatible helper colorings $\mathcal{F}_l$.
Because each of these collections contains finitely many functions, we will identify $\mathcal{F}_l$ with the direct sum of its members: $\mathcal{F}_l=\bigoplus_{(r_1,\dots,r_i)\in S_{\leq l}}f_{r_1,\dots,r_i}$.

We define $\mathcal{F}_n$ by induction on $l$.  
Let $\mathcal{F}_1=f_n$.  Then $\mathcal{F}_1$ is computable because $f=f_n$ is computable.
Trivially, it follows that $\mathcal{F}_1$ is $low$ and is $\zero'$-computable. 
At stage $l+1$, we extend a collection $\mathcal{F}_l$ of length $\leq l$ compatible helper colorings to a collection $\mathcal{F}_{l+1}$ of length $\leq l+1$ compatible helper colorings.

Suppose $l$ satisfies $n>l\geq 2$, 
	and that we have chosen $\mathcal{F}_{l-1}=\bigoplus_{j\leq l-1}p_j$ 
	to be $low^{\emptyset^{(l-2)}}$, 
	where $p_{j}$ is a path through $[T_{j}]$.
Define $T_l$ using $\mathcal{F}_{l-1}$ as above, 
	and note that $T_l$ is infinite by Claim \ref{claim.Tl-inf}.
Because $T_{l}$ is $\Pi^{0,\mathcal{F}_{l-1}}_2$, 
	Claim \ref{claim.Pi2-to-Sigma1} gives a $\Sigma^{0,\mathcal{F}_{l-1}}_1$ tree $S_l$ such that\ $[T_l]=[S_l]$.  
Because $\mathcal{F}_{l-1}$ is $low^{\emptyset^{(l-2)}}$, 
	$S_l$ is $\emptyset^{(l-1)}$-computable 
	and there is a $low^{\emptyset^{(l-1)}}$ path $p_l\in[S_l]=[T_l]$.
Setting $\mathcal{F}_l=\mathcal{F}_{l-1}\oplus p_l$, 
	we see that $\mathcal{F}_l$ is $low^{\emptyset^{(l-1)}}$ and is therefore $\zero^{(l)}$-computable.

Finally, suppose $l=n$, and 
	that we have chosen $\mathcal{F}_{n-1}=\bigoplus_{j\leq n-1}p_j$ to be $low^{\emptyset^{(n-2)}}$, 
	where $p_{j}\in[T_{j}]$.
Define $T_n$ using $\mathcal{F}_{n-1}$ as above, 
	and note that $T_n$ is infinite by Claim \ref{claim.Tl-inf}.
Because $T_n$ is $\Pi^{0,\mathcal{F}_{n-1}}_2$, Claim \ref{claim.Pi2-to-Sigma1} 
	gives a $\Sigma^{0,\mathcal{F}_{n-1}}_1$ tree $S_n$ such that\ $[T_n]=[S_n]$. 
Because $\mathcal{F}_{n-1}$ is $low^{\emptyset^{(n-2)}}$, 
	the tree $S_n$ is $\emptyset^{(n-1)}$-computable, 
	and therefore $\paset$ computes some path $p_n\in[S_n]=[T_n]$.  

Set $\mathcal{F}_n=\mathcal{F}_{n-1}\oplus p_n$.  
Then, by definition of each $T_l$,  $\{f_{r_1,\dots,r_l} : (r_1,\dots,r_l)\in\partitions$ and $f_{r_1,\dots,r_l}$ is the $(r_1,\dots,r_l)^{th}$ component of $\mathcal{F}_n\}$ is a $P$-computable collection of length $\leq n$ compatible helper colorings.
Applying \ref{lemma.PRTn.use-colorings}, 
	we obtain a $\paset$-computable set $A$ that is packed for $\phi$ and semi-homogeneous for $f$. 
\end{proof}

\begin{corollary}
\label{cor.PRTn.arith}
Fix $n\in\omega$.  Each computable instance of $\PRT^n_k$ has a $\Delta^0_{n+1}$ definable solution.
\end{corollary}
\begin{proof}
Recall that $\emptyset^{(n)}$ is $\Delta^0_{n+1}$ and that $\zero^{(n)}\gg\zero^{(n-1)}$.
\end{proof}

Formalizing the above construction in second order arithmetic, we obtain a reverse mathematics analog.

\begin{corollary}
\label{cor.PRTn.ACA}
$\ACA$ implies $\PRT^n_k$ over $\RCA$.
\end{corollary}
\begin{proof}
Fix any $n\in\omega,k\in\N$.
Because $\ACA$ implies that $\zero^{(n)}$ exists and because and $\zero^{(n)}\gg\zero^{(n-1)}$, the set constructions used to prove Theorem \ref{thm.PRTn.pa} can be performed in $\ACA$.  
We leave it to the reader to confirm that the verifications can be performed using induction for arithmetical formulas.
\end{proof}

\section{Lower bounds and reversals}
\label{sect.reversals}

In this section, we give lower bounds on the strength of $\PRT^n_k$.
We first prove that $\PRT^n$ implies $\RT^n$ over $\RCA$. 
Modifying this argument, we show that for each $n$, there is a computable instance of $\PRT^n_{2^{n-1}+1}$ with no $\Sigma^0_n$ solution.
The key tool is both proofs is Theorem \ref{thm.EG.sharpcolors}, which is essentially Theorem 2.3 of \cite{ramseytype}.

The first step in showing that $\PRT^n$ implies $\RT^n$ over $\RCA$ is to state and prove a version of Theorem \ref{thm.EG.sharpcolors} appropriate for reverse mathematics.

\begin{defn} 
We say that $\phi:\N\rightarrow\N$ is an \emph{order function} if $\phi$ is total, non-decreasing, and has unbounded range.
\end{defn}

The most natural choice for $\phi$ in $\PRT$ is an order function.

\begin{remark}[$\RCA$] 
Fix $n\in\omega,k\in\N$.  
For each $w$, let $\phi_{max}(w)$ be the largest $m$ such that $w\rightarrow (m)^n_{k+1}$. 
Then $\phi_{max}$ is a total, $\Delta^0_1$ definable order function.
\end{remark} 
\begin{proof} 
Clearly $\phi_{max}$ is total, $\Delta^0_1$ definable, and non-decreasing.  
Finite Ramsey's theorem, which is provable in $\RCA$, implies that $\phi_{max}$ has unbounded range. 
\end{proof}

Recall that for exponent $n$, we write $\partitions$ for the set of all partition types $(r_1,\dots,r_l)$ such that $r_1+\dots+r_l=n$.  
We will write $\allones$ to refer to the partition type where $r_i=1$ for each $i$.  That is, $\allones=(1,\dots,1)$. 

Recall also that for each increasing sequence $\{w_i\}$ and each set $X\in[\N]^n$, 
	we say that $(r_1,\dots,r_l)$ is the partition type of $X$ with respect to $\{(w_i,w_{i+1}]\}$ 
	if there are $j_1<\dots<j_{l}$ such that $|X\cap(w_{j_i},w_{j_i+1}]| = r_i$ for each $i\leq l$.

We now prove our lemma, which is an adaptation of the proof of Theorem \ref{thm.EG.sharpcolors}.
\begin{lemma}[$\RCA$]
\label{lemma.PRTn-sharp}
Fix $n\in\omega$. Let $\phi:\mathbb{N}\rightarrow\mathbb{N}$ be any order function. 
There is a coloring $g:[\mathbb{N}]^n\rightarrow \partitions$
	and a strictly increasing function $i\mapsto w_i$ such that\
\begin{itemize}
	\item 
	$g(X)$ is the partition type of $X\in[\mathbb{N}]^n$ with respect to $\{(w_i,w_{i+1}]\}$, and 
	\item 
	for any infinite $A\subseteq \mathbb{N}$, 
		either $A$ is sparse for $\phi$ or $\{g(X):X\in[A]^n\} = \partitions$.
\end{itemize}
\end{lemma}
\begin{proof}
We define $w_i$ by induction on $i$.  Let $w_1=1$.  
For $i>1$, define $w_i$ to be the least element of $\{w>w_{i-1} : \phi(w)\geq n\cdot i\}$. 
This set is nonempty because $\phi$ has unbounded range, and has a least element by $\Delta^0_1$ induction.
We have defined $i\mapsto w_i$ by iterating a total $\Delta^0_1$ function, so the map is total by $\Sigma^0_1$ induction (and Proposition 6.5 of \cite{combprinciples}).

For each $X\in[\N]^n$, define $g(X)$ to be the partition type of $X$ with respect to the sequence $\{(w_i,w_{i+1}]\}$.  
Then $g$ and $i\mapsto w_i$ have $\Delta^0_1$ definitions, so exist by $\Delta^0_1$ comprehension.

We must verify that $g$ assigns all colors to any packed set.
Fix any set $A=\{a_1<a_2<\dots\}$.  
Note that if there are $n$ values of $i$ such that $|A\cap(w_i,w_{i+1}]|\geq n$, 
	then $\{g(X) : X\in[A]^n\}=\partitions$, and $g$ assigns all colors to $A$.

Suppose that $A$ is not given all colors by $g$. 
Then there is some $\hat{i}$ such that $(\forall i\geq\hat{i})[|A\cap(w_i,w_{i+1}]|<n]$.
We will show that $A$ is sparse by defining $i_0\geq\hat{i}$ 
	such that for all $i\geq i_0$, $|A\cap\{1,\dots,w_{i+1}\}|<n\cdot i$. 
If $\hat{m}:=|A\cap\{1,\dots,w_{\hat{i}}\}| < n\cdot \hat{i}$, set $i_0=\hat{i}$.
Otherwise, if $\hat{m}\geq n\cdot\hat{i}$, 
	there are $\hat{m}-n\cdot\hat{i}$ more elements in $A\cap\{1,\dots,w_{\hat{i}}\}$ than desired.
Set $i_0=\hat{i}+(\hat{m}-n\cdot\hat{i})$.
For each $i$ between $\hat{i}$ and $i_0$, $|A\cap\{1,\dots,w_i\}|$ increases by at most $n-1$, while $n\cdot i$ increases by $n$. 
Therefore, once $i$ is at least $i_0$, we have  $|A\cap\{1,\dots,w_i\}|<n\cdot i$, as desired.

Because $A=\{a_1<a_2<\dots\}$ and $\{w_1<w_2<\dots\}$ are infinite and $\Delta^0_1$, 
	there is some $j_0$ such that for each $j\geq j_0$, there is an $i\geq i_0$ such that\ $a_j\in (w_i,w_{i+1}]$.  

Fix any $j\geq j_0$.  
Then $n\cdot i>|A\cap\{1,\dots,w_{i+1}\}|$ by definition of $j_0$. 
Recall that we defined $w_i$ so that $\phi(w_i)\geq n\cdot i$.
Putting it all together, because $\phi$ is non-decreasing and $a_j\geq w_i$ we see that $\phi(a_j)\geq\phi(w_i)\geq n\cdot i > |A\cap\{1,\dots,w_{i+1}\}|\geq |A\cap\{1,\dots,a_j\}|$.
That is, $\phi(a_j)>|A\cap\{1,\dots,a_j\}|$.

Because $\phi(a_j)>|A\cap\{1,\dots,a_j\}|$ for all but finitely many $j$, 
	and because $\phi$ is non-decreasing, 
	it follows that $A$ is sparse for $\phi$.
It follows that every set is either sparse or given all colors by $g$.
\end{proof}

For each packed set $X$, the function $g$ assigns all $2^{n-1}$ colors in $\partitions$ to $[X]^n$.
Given a coloring $f:[\N]^n\rightarrow\{1,\dots,k\}$, 
	we will adapt $g$ slightly to obtain a function 
	$h:[\mathbb{N}]^n\rightarrow (\partitions\backslash \{\allones\}) \sqcup \{1,\dots,k\}$.  
Applying $\PRT^n_{2^{n-1}-1+k}$ to $h$ and $\phi_{max}$ 
	will produce a packed set $A$ s.t.\ every $Z\in[A]^n$ with partition type $\allones=(1,\dots,1)$ 
	is given a single color by $f$.  
We will then refine $A$ to obtain an infinite $f$-homogeneous set $H\subset A$ by putting a single element of each block into $H$.

Without loss of generality, we will always assume that $\partitions\cap\{1,\dots,k\}=\emptyset$.

\begin{claim}[$\RCA$] 
\label{claim.reversals.h} 
Fix $f:[\N]^n\rightarrow\{1,\dots,k\}$ and $\phi$ an order function as in $\PRT^n_k$, and let $g$ be the function obtained in Lemma \ref{lemma.PRTn-sharp}.
For each $Z\in[\N]^n$, we define 
$$h(Z)=	\begin{cases}  
			f(Z)	& \text{ if } g(Z)=\allones,\\
			g(Z)	& \text{ otherwise.}
		\end{cases}
$$
Let $A$ be semi-homogeneous for $h$ and packed for $\phi$.  
Then there is a unique $\hat{c}\in\{1\dots,k\}$ s.t.\ $f(X)=\hat{c}$ for each $X\in [A]^n$ with $g(X)= \allones$.
\end{claim}
\begin{proof} 
	We first examine the colors assigned to $A$ by the coloring $g:[\N]^n\rightarrow \partitions$.
	Because $A$ is packed for $\phi$, it follows that $g$ assigns all possible colors to subsets of $A$.
	For each $c\in \partitions\backslash\{\allones\}$, select some $X_c\in[A]^n$ such that $g(X_c)=c$. 

	We now examine the colors assigned to $A$ by the helper coloring $h:[\N]^n\rightarrow \{1,\dots,k\} \sqcup \partitions\backslash\{\allones\}$.
	For each $c\in \partitions\backslash\{\allones\}$, our definition of $h$ implies that $h(X_c)=g(X_c)=c$.
	Recall that $|\partitions\backslash\{\allones\}|=2^{n-1}-1$.
	Also by our definition of $h$, we know that $h(X)\in\{1,\dots,k\}$ for each $X$ with $g(X)=\allones$.  
	Because $A$ is semi-homogeneous for $h$, and because $\partitions\cap\{1,\dots,k\}=\emptyset$, it follows that there is a unique color $\hat{c}\in\{1,\dots,k\}$ such that $h(X)=\hat{c}$ for each $X\in [A]^n$ with $g(X)=\allones$.  

	Examining our definition of $h$ one last time, we see that whenever $g(X)=\allones$, we have $h(X)=f(X)$.  Consequently, we have shown that there is a unique color $\hat{c}\in\{1,\dots,k\}$ such that for each $X\in [A]^n$ with $g(X)=\allones$, $f(X)=\hat{c}$.  
\end{proof}

\begin{theorem}[$\RCA$]
\label{thm:PRTprovesRT}
$\PRT^n_{2^{n-1}-1+k}$ implies $\RT^n_k$ for each $n\in\omega$ and $k\in\N$.
\end{theorem}
\begin{proof}

Suppose $\PRT^n_{2^{n-1}-1+k}$ holds. 
Given a function $f:[\mathbb{N}]^n\rightarrow\{1,\dots,k\}$, we must produce an infinite set $H$ homogeneous for $f$.  

Recall that $\phi_{\max}(w)=\max m[w\rightarrow(m)^n_{k+1}]$ is a $\Delta^0_1$ definable order function.
Applying Lemma \ref{lemma.PRTn-sharp} with $\phi=\phi_{max}$, 
	we obtain $g:[\mathbb{N}]^n\rightarrow \partitions$ s.t.\ $g(X)$ is the partition type of $X\in[\N]^n$.  
Define the coloring $h:[\mathbb{N}]^n\rightarrow (\partitions\backslash \{\allones\}) \sqcup \{1,\dots,k\}$ as in Claim \ref{claim.reversals.h}.
Note that $h$ is a coloring of $[\N]^n$ into $2^{n-1}-1+k$ colors.  
Furthermore, $h$ has a $\Delta^0_1$ definition, so exists by $\Delta^0_1$ comprehension.
By $\PRT^n_{2^{n-1}-1+k}$, there is a set $A$ that is semi-homogeneous for $h$ and packed for $\phi_{max}$.  

Let $H=\bigcup_{i\in \mathbb{N}} \{\min(A\cap (w_i,w_{i+1}])\}$.
For each $i$, $A\cap(w_i,w_{i+1}]$ is $\Delta^0_1$ with parameters, so it has a least element.
Because $A$ is infinite and each interval is finite, $H$ is infinite.  
Clearly, $H$ is $\Delta^0_1$ definable from $A$, so exists by $\Delta^0_1$ comprehension.

Suppose $X\in [H]^n$.  
By definition of $H$, at most $1$ element of $X$ is in any interval $(w_i,w_{i+1}]$.  
Then $X$ has partition type $\allones$, so $g(X)=\allones$.
Let $\hat{c}$ be the unique color in Claim \ref{claim.reversals.h}.  
Because $H\subseteq A$, it follows that $f(X)=\hat{c}$.  

In summary, $H$ is infinite and $f$-homogeneous with color $\hat{c}$, as desired.
\end{proof}

We conclude this section by showing that for $n\geq 2$ and $k>2^{n-1}$, 
	there is a computable instance of $\PRT^n_k$ with no $\Sigma^0_n$ solution.
We use the following result.

\begin{theorem}[Jockusch \cite{J72}] \label{thm.RT.no-sigma-n}
For each $n\geq2$, there exists a computable coloring $f:[\N]^n\rightarrow\{1,2\}$ such that no $\Sigma^0_n$ set is homogeneous for $f$. 
\end{theorem}

We can now prove an arithmetical lower bound on $\PRT^n_k$.

\begin{theorem}
\label{thm.PRTn.no-sigma}
If $n\geq 2$ and $k>2^{n-1}$, there is a computable instance of $\PRT^n_k$ such that\ no $\Sigma^0_n$ definable set is both packed for $\phi$ and semi-homogeneous for $f$.
\end{theorem}
\begin{proof}
Suppose toward a contradiction that for each appropriate computable coloring and $\phi$, there is a $\Sigma^0_n$ definable set which is both packed and semi-homogeneous.

Let $f:[\N]^n\rightarrow\{1,\dots,k\}$ be a coloring with no $\Sigma^0_n$ definable homogeneous set, which exists by Theorem \ref{thm.RT.no-sigma-n}.  
Note that the function $g$ from Lemma \ref{lemma.PRTn-sharp} with $\phi=\phi_{max}$ is computable.
Define $h$ as in Claim \ref{claim.reversals.h}.
Then $h$ is computable because it is computable from $f$, $g$, and $\phi_{max}$.

Suppose that $A$ is a $\Sigma^0_n$ set that is packed for $\phi_{max}$ and semi-homogeneous for $h$.  
Then there is a $\Delta^0_n$ formula $\theta$ such that $x\in A \iff (\exists y)[\theta(x,y)]$.
Let 
	$$x\in H \iff (\exists y)(\exists i)[\theta(x,y) \land (x\in(w_i,w_{i+1}]) \land (\forall z\in(w_i,x))(\forall t\leq y)[\neg\theta(z,t)]].$$  
Note that this is a $\Sigma^0_n$ definition for $H$.

Because $A$ is infinite, and because each interval $(w_i,w_{i+1}]$ is finite, we see that $H$ is infinite. 
Because each element of $[H]^n$ has partition type $\allones$ with respect to the sequence $\{(w_i,w_{i+1}]\}$, it follows that $H$ is $g$-homogeneous with color $\allones$.  
By Claim \ref{claim.reversals.h}, there is a unique $\hat{c}$ such that $f(X)=\hat{c}$ for each $X\in[H]^n$.
In short, $H$ is an infinite $f$ homogeneous set that is $\Sigma^0_n$, contradicting our choice of $f$.
\end{proof}

\section{Summary of results}
\label{sect.summary}

\begin{defn}
Given $n,k\in\N$, a \emph{computable instance of $\PRT^n_k$} is computable coloring $f:[\N]^n\rightarrow\{1,\dots,k\}$ and a computable $\phi:\N\rightarrow\N$ such that\ $w\rightarrow(\phi(w))^n_{k+1}$ for all $w$.
We say that $A\subseteq\N$ is a \emph{solution} to a computable instance of $\PRT^n_k$ if $A$ is packed for the appropriate $\phi$, and semi-homogeneous for the appropriate $f$.
\end{defn}

From the perspective of computability theory, we have shown:
\begin{theorem} Fix $n,k\in\N$.  
\label{thm.comp-summary}
\begin{enumerate}
	\item \label{thm.comp-summary.pa}
		For any $\paset\gg\zero^{(n-1)}$, each computable instance of $\PRT^n_k$ 
		has a $\paset$-computable solution.  Hence, there is always a $\Delta^0_{n+1}$ solution.
	\item \label{thm.comp-summary.sigma} 
		If $n\geq 2$, there is a computable instance of $\PRT^n_{2^{n-1}+1}$ with no $\Sigma^0_n$ solution. 
	\item  \label{thm.comp-summary.low}
		Any computable instance of $\PRT^2_k$ has a $low_2$ solution.
	\item \label{thm.comp-summary.comp}
		Each computable instance of $\PRT^1_k$ has a computable solution.
\end{enumerate}
\end{theorem}
\begin{proof}
\eqref{thm.comp-summary.pa} is Theorem \ref{thm.PRTn.pa} and Corollary \ref{cor.PRTn.arith}.  
\eqref{thm.comp-summary.sigma} follows from Theorem \ref{thm.PRTn.no-sigma}.
\eqref{thm.comp-summary.low} is Theorem \ref{thm.PRT2.low2}, and 
\eqref{thm.comp-summary.comp} is Corollary \ref{cor.PRT1.comp}.
\end{proof}

From the perspective of reverse mathematics, we have shown:
\begin{theorem} Over $\RCA$,
\label{thm.rm-summary}
\begin{enumerate}
	\item \label{thm.rm-summary.eq}
		$\PRT^n_k$ is equivalent to $\RT^n_k$ for $n\in\omega$ s.t.\ $n\neq 2$ 
		and $k\in\N$ s.t.\ $k>2^{n-1}$, 
	\item \label{thm.rm-summary.imp}
		$\PRT^2_{k+1}$ implies $\RT^2_k$ for each $k\in\N$, 
	\item \label{thm.rm-summary.notimp}
		$\PRT^2_k$ does not imply $\ACA$ for any $k\in\omega$, and
	\item \label{thm.rm-summary.bsig2}
		$(\forall k)\PRT^1_k$ is equivalent to B$\Sigma^0_2$.
\end{enumerate}
\end{theorem}
\begin{proof}
To prove \eqref{thm.rm-summary.eq}, we work over $\RCA$.
First, consider $n=1$.  
For any $k\in\mathbb{N}$, $\RT^1_k$ implies $\PRT^1_k$ by Theorem \ref{thm.RT1k->PRT1k}, and $\PRT^1_{k}$ implies $\RT^1_k$ by Theorem \ref{thm:PRTprovesRT}.

Next, consider any $n\in\omega$ with $n\geq 3$ and any $k\in\N$ with $k>2^{n-1}$.
Because $n\geq3$, $\RT^n_k$ is equivalent to both $\RT^n_2$ and $\ACA$.
By Theorem \ref{thm:PRTprovesRT}, $\PRT^n_k$ implies $\RT^n_2$, so $\PRT^n_k$ also implies $\RT^n_k$.  
By Corollary \ref{cor.PRTn.ACA},  $\ACA$ implies $\PRT^n_k$, so $\RT^n_k$ also implies $\PRT^n_k$.

\eqref{thm.rm-summary.imp} is Theorem \ref{thm:PRTprovesRT} for exponent $n=2$, and 
\eqref{thm.rm-summary.notimp} follows from Corollary \ref{cor.PRT2.notimp}.
For \eqref{thm.rm-summary.bsig2}, recall that $(\forall k)\RT^1_k$ is equivalent to B$\Sigma^0_2$ by a result of \cite{hirst-thesis}.
\end{proof}

\section{Further questions}
\label{sect.questions}

In this paper, we have shown that packed Ramsey's theorem is close in strength to Ramsey's theorem.  
We close with a number of questions concerning the exact strength of this packed Ramsey's theorem.

\begin{question}
Does each computable instance of $\PRT^n_k$ have a $\Pi^0_{n}$ solution?
\end{question}

\begin{remark}
In many of our proofs of $\PRT^n_k$, we have echoed proofs of $\RT^n_k$, replacing the construction of a `homogeneous set of numbers' with the construction of `a sequence of blocks whose induced colorings are homogeneous.'  
In each of these adaptions, care was required when selecting the finite blocks, since they must be colored appropriately with a large set of numbers. 
\par
In \cite{J72}, Jockusch proved that each computable instance of $\RT^n_k$ has a $\Pi^0_n$ solution.  
Unfortunately, adapting Jockusch's proof using these methods does not appear to produce $\Pi^0_2$ solutions to instances of $\PRT^2_k$.
The main obstacle can be expressed this way: the construction must somehow avoid selecting blocks that contain two numbers $x_1$ and $x_2$ such that each $x_i$ is given color $i$ with all large enough numbers. 
\end{remark}

There are a number of open questions about the precise reverse mathematical strength of $\PRT^2_k$. 
The following question is of particular interest.

\begin{question}
Does $\RT^2_2$ imply $\PRT^2_3$ over $\RCA$?
\end{question}

Absent an equivalence of $\PRT^2_k$ and $\RT^2_2$, can upper bounds of $\RT^2_2$ be extended to $\PRT^2_k$? 
For example, Cholak, Jockusch, and Slaman showed in \cite{CJS} that $\RT^2_2$ is $\Pi^1_1$ conservative over $\ISigma^0_2$.

\begin{question}
Is $\PRT^2_k$ $\Pi^1_1$ conservative over $\ISigma^0_2$? 
Is it conservative over $\RT^2_2$?
\end{question}

Recall that $\WKL$ asserts that each infinite binary tree has an infinite path.
In \cite{liu}, Liu showed that $\RT^2_2$ does not imply $\WKL$.

\begin{question}
Does $\PRT^2_3$ imply $\WKL$ over $\RCA$?
\end{question}

We close with a natural combinatorial question.  
Recall that the restriction on the rate of growth of $\phi$ was used primarily 
	to prove that $[\N]^{n-1}$ is large.
Can weaken the restriction on $\phi$, either by modifying the notion of largeness or by giving an alternate construction?

\begin{question}
\label{question.improvephi}
Fix $n,k$.
Are there functions $\phi$ s.t.\ $w\not\rightarrow(\phi(w))^n_{k+1}$ for infinitely many $w$ such that 
for all $f:[\N]^n\rightarrow\{1,\dots,k\}$, there is a set $A$ which is packed for $\phi$ and semi-homogeneous for $f$? 

Conversely, for each $\phi$ s.t.\ $w\not\rightarrow(\phi(w))^n_{k+1}$ infinitely often, is there is a coloring $f:[\N]^n\rightarrow\{1,\dots,k\}$ s.t.\ every set semi-homogeneous for $f$ is also sparse for $\phi$?
\end{question}

Although exact values for the fastest growing $\phi$ 
	such that $w\rightarrow(\phi(w))^n_{k+1}$ are not known when $n\geq 2$, 
	some upper and lower bounds are known.
For $n\in\N$, $log_{n-1}$ denotes the $n-1$-iterated logarithm.
Theorem 26.6 of \cite{comb-set-thy} says that for each $n,k\geq 2$, 
	there is a $\hat{c}_{n,k+1}$ such that 
	the function $\phi(w)= \hat{c}_{n,k+1}\cdot\log_{n-1}w$ 
	satisfies the conditions of $\PRT^n_k$ for all large enough $w$.  
When $n,k\geq 3$, Theorem 26.3 of \cite{comb-set-thy} says that this lower bound is sharp up to a multiplicative constant.
More precisely, fix any $n,k\geq 3$, and consider the greatest function $\phi$ as in $\PRT^n_k$.  
Then Theorem 26.3 of \cite{comb-set-thy} says that there is a $d_n$ s.t.\ $\phi(w) < d_n\cdot\log_{n-1}w$ for all large enough $w$.  

This gives a sharpening of question \ref{question.improvephi}.
\begin{question}
Fix $n\geq 3$.  
Is there a function $\phi$ with $\phi(w)>d_n\cdot \log_{n-1}(w)$ 
	such that for all $f:[\N]^n\rightarrow\{1,\dots,k\}$, 
	there is a set $A$ which is packed for $\phi$ and semi-homogeneous for $f$? 
\end{question}

\bibliographystyle{amsplain}
\bibliography{../../../BibTeX/flood-bibliography}
\nocite{CJS-correction}

\end{document}